\documentclass[reqno]{amsart}
\usepackage{amssymb}
\usepackage{graphicx}
\usepackage{amsmath}

\usepackage{yfonts}
\usepackage{pbsi}

\usepackage[usenames, dvipsnames]{color}
\usepackage{verbatim}
\usepackage{mathrsfs}
\usepackage{bm}
\usepackage{cite}

\numberwithin{equation}{section}

\newtheorem{theorem}{Theorem}[section]
\newtheorem{corollary}[theorem]{Corollary}
\newtheorem{lemma}[theorem]{Lemma}
\newtheorem{prop}[theorem]{Proposition}

\theoremstyle{definition}
\newtheorem{remark}[theorem]{Remark}

\theoremstyle{definition}

\theoremstyle{definition}

\makeatletter
\def\dashint{\operatorname%
{\,\,\text{\bf-}\kern-.98em\DOTSI\intop\ilimits@\!\!}}
\makeatother

\def\\det{\text{det}}

\def\.5{\frac{1}{2}}

\newcommand{\va}{\ensuremath{\varepsilon}}
\newcommand{\ptl}{\ensuremath{\partial}}
\newcommand{\om}{\ensuremath{\Omega}}

\newcommand{\R}{\ensuremath{\mathbb{R}}}
\newcommand{\bt}{\ensuremath{\beta}}

\newcommand{\lam}{\ensuremath{\lambda}}
\newcommand{\alp}{\ensuremath{\alpha}}
\newcommand{\dt}{\ensuremath{\delta}}

\newcommand{\RN}[1]{%
  \textup{\uppercase\expandafter{\romannumeral#1}}%
}

\newcommand{\Div}{\operatorname{div}}
\newcommand{\dist}{\text{dist}}

\renewcommand{\epsilon}{\varepsilon}

\newcounter{marnote}

\begin{document}


\title[Gradient Estimates for Perfect Conductivity Problem ]{The Optimal Gradient Estimates for Perfect Conductivity Problem with $C^{1, \alp}$ inclusions}

\author[Y. Chen]{Yu Chen}
\address[Y. Chen]{School of Mathematical Sciences, Beijing Normal University, Laboratory of Mathematics and Complex Systems, Ministry of Education, Beijing 100875, China. }
\email{chenyu@amss.ac.cn.}

\author[H.G. Li]{Haigang Li}
\address[H.G. Li]{School of Mathematical Sciences, Beijing Normal University, Laboratory of Mathematics and Complex Systems, Ministry of Education, Beijing 100875, China. }
\email{hgli@bnu.edu.cn. }

\author[L.J. Xu]{Longjuan Xu}
\address[L.J. Xu]{School of Mathematical Sciences, Beijing Normal University, Laboratory of Mathematics and Complex Systems, Ministry of Education, Beijing 100875, China.}
\email{ljxu@mail.bnu.edu.cn. }

\maketitle

\begin{abstract}

In high-contrast composite materials, the electric field concentration is a common phenomenon when two inclusions are close to touch. It is important from an engineering point of view to study the dependence of the electric field on the distance between two adjacent inclusions. In this paper, we derive the upper and lower bounds of the gradient of solutions to the conductivity problem where two perfectly conducting inclusions are located very close to each other. To be specific, we extend the known results of Bao-Li-Yin (ARMA 2009) in two folds: First, we weaken the smoothness of the inclusions from $C^{2, \alp}$ to $C^{1, \alp}$. To obtain an pointwise upper bound of the gradient, we follow an iteration technique developed by Bao-Li-Li (ARMA 2015), who mainly deal with the system of linear elasticity. However, when the inclusions are of $C^{1, \alp}$, we can not use $W^{2,p}$ estimates for elliptic equations any more. In order to overcome this new difficulty, we take advantage of De Giorgi-Nash estimates and Campanato's approach to apply an adapted version of the iteration technique with respect to the energy. A lower bound in the shortest line between two inclusions is also obtained to show the optimality of the blow-up rate. Second, when two inclusions are only convex but not strictly convex, we prove that blow-up does not occur any more.  Moreover, the establishment of the relationship between the blow-up rate of the gradient and the order of the convexity of the inclusions reveals the mechanism of such concentration phenomenon. 

\vspace{.5cm}




\end{abstract}

\section{Introduction}

\subsection{Background}

Let $D$ be a bounded open set in $\mathbb R^{n}$, $n\geq2$,  $D_1$ and $D_2$ be its two adjacent subdomains with $\va$-apart. The perfect conductivity problem is modeled as follows:
\begin{equation}\label{equ-infty-interior}
\begin{cases}
\Delta{u}=0& \mbox{in}~D\setminus\overline{D_{1}\cup D_{2} },\\
u=C_{i} & \mbox{on}~\ptl D_i, ~i=1,2, \\
\int_{\partial{D}_{i}}\frac{\partial{u}}{\partial\nu}\big|_{+}=0& i=1, \, 2,\\
u=\varphi&\mbox{on}~\partial{D},
\end{cases}
\end{equation}
where $\varphi \in C^{1, \alp}(\ptl D)$, $\alp \in (0,\, 1)$ and
$$\frac{\partial{u}}{\partial\nu}\bigg|_{+}:=\lim_{\tau\rightarrow0}\frac{u(x+\nu\tau)-u(x)}{\tau}.$$
Here and throughout this paper $\nu$ is the outward unit normal to the domain and the subscript $\pm$ indicates the limit from outside and inside the domain, respectively. In the second line of \eqref{equ-infty-interior}, the constants $C_{1}$ and $C_{2}$ represent free constant boundary conditions. The third line of \eqref{equ-infty-interior} means that there is no flux through the boundaries of inclusions. By a variational argument, there exists a unique pair of $C_{1}$ and $C_{2}$ such that \eqref{equ-infty-interior} has a solution $u\in H^{1}(\Omega)$, see e.g. \cite{bly1}. Namely, it can be described as the unique function which has the least energy in appropriate functional space, that is, 
$$E[u]=\min_{v\in\mathcal{A}}E[v], ~\mbox{where}~\mathcal{A}:=\{v\in H^{1}(\Omega)~|~\nabla v=0~\mbox{in}~{D}_{1}\cup D_2,~v|_{\partial\Omega}=\varphi\}.$$

A simple, two dimensional example, which very well illustrates the main feature of our estimates, would have the domain $D\subset\mathbb{R}^{2}$ model the cross-section of a fiber-reinforced composite, $D_1$ and $D_2$ represent the cross-sections of the stiff fibers, and the remaining subdomain represents the matrix medium. The gradient of the potential $u$ represents  the electrical field in the conductivity problem and the stress in anti-plane elasticity. From the second and third lines of \eqref{equ-infty-interior}, there are constant $C_{1}$ and $C_{2}$ such that $u=C_{i}$ on $\partial{D}_{i}$, $i=1,2$, which are free boundary conditions. This model can also be used to describe many other engineering and physical problems. Since they are mathematically identical henceforth we here use the conductivity terminology.
  
It is well known that the high concentration phenomenon of extreme electric field or mechanical loads in the extreme loads will be amplified by the composite microstructure, for example, the narrow region between two adjacent inclusions. Therefore, an optimal shape of the inclusions is the aim that an engineer pursues to design a more effective composite. Therefore, there have been many important works on the gradient estimates for strictly convex inclusions, especially for circular inclusions (that is, $2$-convex inclusions). For two adjacent disks with $\varepsilon$ apart, Keller \cite{k1} was the first to compute the effective electrical conductivity for a composite containing a dense array of perfectly conducting spheres of cylinders. In \cite{basl}, Babu\v{s}ka et al. numerically analyzed the initiation and growth of damage in composite materials, in which the inclusions are frequently spaced very closely and even touching. Bonnetier and Vogelius \cite{bv} and Li and Vogelius \cite{lv} proved the uniform boundedness of $|\nabla{u}|$ regardless of $\varepsilon$ provided that the conductivities stay away from $0$ and $\infty$. Li and Nirenberg \cite{ln} extended the results in \cite{lv} to general divergence form second order elliptic systems including systems of elasticity. 

For the perfect conductivity problem, the gradient's blow-up feature has attracted much attention in recent years due to its various applications. Much effort has been devoted to understanding of this blow-up mechanics. Ammari, Kang, and Lim \cite{akl} were the first to study the case of the close-to-touching regime of two circular particles whose conductivities degenerate to $\infty$ or $0$, a lower bound on $|\nabla u|$ was constructed there showing it blows up in both the perfectly conducting and insulating cases. This blow-up was proved to be of order $\varepsilon^{-1/2}$ in $\R^2$. In their subsequent work with H. Lee and J. Lee \cite{aklll}, they established upper and lower bounds to show the blow-up rate $\va^{-1/2}$ is optimal in $\R^2$. Subsequently, it has been proved by many mathematicians that for the perfect conducting case the generic blow-up rate of $|\nabla{u}|$ is $\varepsilon^{-1/2}$ in dimension two, $|\varepsilon\ln\varepsilon|^{-1}$ in dimension three, and $\varepsilon^{-1}$ in dimensions equal and greater than four. See Bao, Li and Yin \cite{bly1,bly2}, as well as Lim and Yun \cite{ly,ly2}, Yun \cite{y1,y2,y3}. The corresponding boundary estimates see \cite{aklll} and \cite{LX}. For the Lam\'{e} system with partially infinitely coefficients, see \cite{bll, bll2,bjl,ky}.

Recently, the characterizations of the singular behavior of $\nabla u$ for the perfect case was further developed in \cite{ackly,akllz,bt1,bt2,kly,jlx,kly2,ll-p,ky2}. The stress blow-up in the hole case has been characterized by an explicit function in Lim and Yu \cite{lyu}. The $C^1$, $C^2$ estimates for the elliptic equations with coefficients having Dini mean oscillation condition was established in \cite{dk,dek}.  For more related work on elliptic equations and systems from composites, see \cite{adkl,bc,dong,dongli,dongzhang,kly0,llby,m} and the references therein.

In this paper, we mainly prove that in perfect conductivity problem the blow-up rates of the electric field, $|\nabla u|$, are totally determined by the geometry of the inclusions. This geometry quantity is the order of the convexity of the inclusions, we refer to it as $m$-convexity. For example, the circular inclusions are $2$-convex, as mentioned before. For all the known results for perfect conductivity problem, the $C^{2,\alpha}$ smoothness of the inclusions is assumed, which means that $m=2+\alpha$, see \cite{bly1,akl} for instance. However, from the classical regularity theory of elliptic partial differential equations, it suffice to assume that the domain is $C^{1, \alp}$, to establish the gradient estimates of the solutions. Therefore, the first contribution of this present paper is to deal with the cases that $1<m<2$. Although from the regularity theory of partial differential equation, the $C^{1,\, \alp}$ smoothness is sufficient to obtain $L^{\infty}$ estimate of the gradient, more new difficulty is encountered to apply the iteration argument developed in \cite{bll,bll2} to establish a pointwise gradient estimate. The reason is that at this moment, the constructed auxiliary function is not smooth enough to employ the $W^{2, \, p}$-estimates as in the case of $C^{2, \alp}$ inclusions (see Proposition \ref{prop1}). Here we make use of more delicate analysis technique, such as De Giorgi-Nash estimates and Campanato's approach, to adapt the iteration technique to make up this gap. On the other hand, as mentioned above, in most of known results the strict convexity of the inclusions are assumed. However, when the inclusions are only convex but not strictly convex (see Figure \ref{fig1}) in the perfect conductivity problem, we find that blow-up does not occur any more, which corresponds the case that $m= \infty$. We prove that $|\nabla u|$ is uniformly bounded with respect to $\varepsilon$ whenever the area of the flat region is bigger than zero (Theorem \ref{thm1}) and show the explicit effect of the flatness of the inclusions. The rest cases, $2\leq m <\infty$, are also considered in such frame, see Theorem \ref{rem-m-geq2} below. Thus, we study the full range of $m$, $1<m \leq \infty$ are systematic studied in this paper.

In what follows, we state our main results in two folds, presented in subsection 1.2 and 1.3, respectively.

\subsection{ $C^{1, \alp}$ inclusions, when  $1< m <2$.}

We first fix our domain and notations. Let $D_{1}^{0}$ and $D_{2}^{0}$ be a pair of (touching at the origin) subdomains of $D$, a bounded open set in $\R^n$, $n\geq 2$,  far away from $\partial D$ and satisfy
$$D_{1}^{0}\subset\{(x',x_{n})\in\mathbb R^{n}~|~ x_{n}>0\},\quad D_{2}^{0}\subset\{(x',x_{n})\in\mathbb R^{n}~|~ x_{n}<0\}.$$
We use superscripts prime to denote the ($n-1$)-dimensional variables and domains, such as $x'$, $B'$ and $\Sigma '$. We assume that $\ptl D_1$ and $\ptl D_2$ are all of $C^{1, \alp}$, $0<\alp<1$. Translate $D_{i}^{0}$ ($i=1,\, 2$) by $\pm\frac{\va}{2}$ along $x_{n}$-axis as follows
\begin{equation}\label{def-D1D2-1}
D_{1}^{\va}:=D_{1}^{0}+(0',\frac{\va}{2})\quad \text{and}\quad D_{2}^{\va}:=D_{2}^{0}+(0',-\frac{\va}{2}).
\end{equation}
For simplicity, we drop the superscript $\va$ and denote
\begin{equation}\label{def-D1D2-2}
D_{i}:=D_{i}^{\va}\, (i=1, \, 2), \quad  \Omega:=D\setminus\overline{D_1\cup D_2},
\end{equation}
and $P_1:= (0',\frac{\va}{2})$, $P_2:=(0',-\frac{\va}{2}) $ be the two nearest points between $\ptl D_1$ and $\ptl D_2$ such that
\begin{equation*}
\text{dist}(P_1, P_2)=\text{dist}(\ptl D_1, \ptl D_2)=\va.
\end{equation*}
We further assume that there exists a constant $R_1$, independent of $\va$, such that the top and bottom boundaries of the narrow region between $\ptl D_1$  and $\ptl D_2$ can be represented, respectively, 
by graphs 
\begin{equation}\label{h1h2'}
x_n=\frac{\va}{2}+h_1(x^\prime)\quad\text{and}\quad x_n=-\frac{\va}{2}+h_2(x^\prime),\quad \text{for}~ |x^\prime|\leq 2R_1,
\end{equation}
where $h_1$, $h_2\in C^{1,\alp}(B'_{2R_{1}}(0^\prime))$ and satisfy 
\begin{equation}\label{h1-h2}
-\frac{\va}{2}+h_{2}(x') <\frac{\varepsilon}{2}+h_{1}(x'),\quad\mbox{for}~~ |x^\prime|\leq 2R_1;
\end{equation}
\begin{equation}\label{h1h1}
h_{1}(0')=h_2(0')=0,\quad\nabla_{x'}h_{1}(0')=\nabla_{x'}h_2(0')=0;
\end{equation}
\begin{equation}\label{h1h_convex2}
\kappa_{0}|x'|^{\alpha}\leq|\nabla_{x'}h_{1}(x')|,|\nabla_{x'}h_2(x')|\leq\,\kappa_{1}|x'|^{\alp},\quad\mbox{for}~~|x'|<2R_{1},
\end{equation}
and 
\begin{equation}\label{h1h3-1}
\|h_{1}\|_{C^{1,\alpha}(B'_{R_{1}})}+\|h_{2}\|_{C^{1,\alpha}(B'_{R_{1}})}\leq{\kappa},
\end{equation}
where the positive constants $\kappa_{0}<\kappa_{1}<\kappa$. Set
\[\Omega_r:=\left\{(x',x_{n})\in \Omega~\big|~ -\frac{\va}{2}+h_2(x')<x_{n}<\frac{\va}{2}+h_1(x'),~|x'|<r\right\}.\]

\begin{theorem}\label{thm-interior}
Let $D_1$, $D_2 \subset D\subset \R^n$ ( $n \geq 2$ ) be two bounded $C^{1,\alp}$ subdomains with  $\va$ apart. Suppose \eqref{h1h2'}--\eqref{h1h3-1} hold. Let $u\in H^1(D)\cap C^1(\overline{\Omega})$ be the solution to \eqref{equ-infty-interior} with $\varphi \in C^{1, \alp}(\ptl D)$.  Then for small $\va >0$, we have
\begin{equation}\label{upper-bounds}
|\nabla u(x',x_{n})| \leq \frac{C\rho_{n,\,\alp}(\va)}{\va+|x'|^{1+\alp}}\cdot\|\varphi\|_{C^{1, \, \alp}(\ptl D)}, \quad \text{for} ~ (x',x_{n})\in \Omega_{R_{1}},
\end{equation}
and
$$\|\nabla u\|_{L^{\infty}(\Omega\setminus\Omega_{R_{1}})}\leq\,C\|\varphi\|_{C^{1, \, \alp}(\ptl D)},$$
where $C$ is a positive constant, independent of $\va$, and 
\begin{equation}\label{def-rho-n-m}
\rho_{n,\,\alp}(\va)=
\begin{cases}
\va^{\frac{\alp}{1+\alp}}& \quad n=2,\\
1& \quad n\geq 3.
\end{cases}
\end{equation}
\end{theorem}

\begin{remark}\label{rem-lower-interior}
To show that the blow-up rates $\va^{\frac{-1}{1+\alp}}$ for $n=2$ and $\va^{-1}$ for $n\geq 3$ are optimal, we also have the lower bound of $|\nabla u(x)|$ on the segment $\overline{P_1 P_2}$, 
\begin{equation*}
|\nabla u(x)|\geq \frac{\rho_{n,\,\alp}(\va)}{C\va}, \quad x\in \overline{P_1 P_2}.
\end{equation*}
For more details, see subsection \ref{part 2 pf lower}.
\end{remark}

\begin{remark}\label{remark-m-to-1}
As $\alp \to 0$ (that is, $m \to 1$), one can see from \eqref{upper-bounds} that $|\nabla u(x)|\leq C\va^{-1}$. But, when $D_1$ and $D_2$ become Lipschitz domains ($m=1$), the corner singularity is another more interesting and challenging topic. See, e.g. Kozlov et al's book \cite{kmr}, Kang and Yun \cite{ky2} for bow-tie structure.
\end{remark}

\subsection{$C^{2, \alp}$ inclusions with partially ``flat'' boundaries,  when $ m =\infty$.}

In this case, we assume $D_1^0$ and $D_2^0$ are two (touching) subdomains of $D$ with $C^{2,\alpha}$ $(0<\alpha<1)$ boundaries and have a part of common boundary $\Sigma'$ with $|\Sigma'|\neq 0$, such that
$$\partial D_{1}^{0}\cap\partial D_{2}^{0}=\Sigma'\subset\mathbb R^{n-1}.$$
See Figure \ref{fig1}. Here, we suppose that $\Sigma'$ is a bounded convex domain in $\mathbb R^{n-1}$, which  can contain an $(n-1)$-dimensional ball, and its center of mass is at the origin.  Translate $D_{i}^{0}$ ($i=1,\, 2$) by $\pm\frac{\va}{2}$ along $x_{n}$-axis as in subsection 1.2 to have $D_1$ and $D_2$ like \eqref{def-D1D2-2}. 

\begin{figure}[t]
\begin{minipage}[c]{0.9\linewidth}
\centering
\includegraphics[width=1.5in]{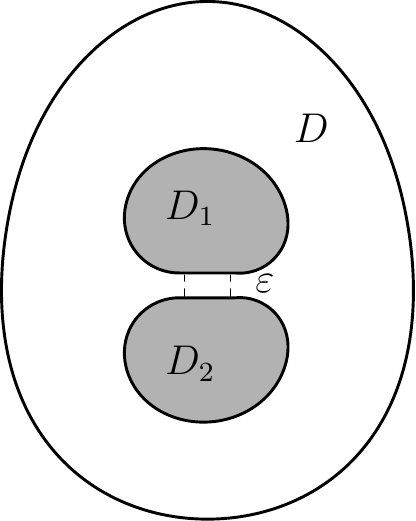}
\caption{\small Two adjacent inclusions with partially ``flat" boundaries.}
\label{fig1}
\end{minipage}
\end{figure}

The top and bottom boundaries of the narrow region between $\partial{D}_{1}$ and $\partial{D}_{2}$ can be represented as follows:  there exists a constant $R_{1}>0$, independent of $\varepsilon$, such that $h_1,\, h_2\in C^{2,\alp}(B^\prime_{2R_1}(0^\prime))$, $\Sigma'\subset B'_{R_{1}}$, satisfying, besides of \eqref{h1h2'}, \eqref{h1-h2}, \begin{equation}\label{h1h2}
h_{1}(x')=h_{2}(x')\equiv0,\quad\mbox{for}~~x'\in\Sigma'.
\end{equation}
\begin{equation}\label{h1---h2}
\nabla_{x'}h_{1}(x')=\nabla_{x'}h_{2}(x')=0,\quad\mbox{for}~~x'\in\partial\Sigma',
\end{equation}
\begin{equation}\label{h1----h2}
\nabla^{2}_{x'}(h_{1}-h_{2})(x')\geq\kappa_{2}I_{n-1},\quad\mbox{for}~~x'\in B'_{R_{1}}\setminus\overline{\Sigma'},
\end{equation}
and
\begin{equation}\label{h1h3}
\|h_{1}\|_{C^{2,\alpha}(B'_{R_{1}})}+\|h_{2}\|_{C^{2,\alpha}(B'_{R_{1}})}\leq{\kappa_3},
\end{equation}
where $\kappa_2,\ \kappa_3$ are positive constants, $I_{n-1}$ is the $(n-1)\times(n-1)$ identity matrix.

\begin{theorem}\label{thm1}
Let $D_{1}, D_{2}\subset{D}\subset\mathbb{R}^{n}$ ($n\geq2$) be two bounded $C^{2,\alp}$ subdomains, with partial flat boundaries $\Sigma'\times\{\pm\frac{\va}{2}\}$, respectively, as in Figure \ref{fig1}. Assume that \eqref{h1h2'}, \eqref{h1-h2} and \eqref{h1h2}--\eqref{h1h3} hold. Let $u\in{H}^{1}(D)\cap{C}^{1}(\overline{\Omega})$ be the solution to \eqref{equ-infty-interior} and $\varphi \in C^{2, \alp}(\ptl D)$. Then for $0<\va<1/2$ and $|\Sigma'|>0$,  we have
\begin{equation}\label{result-bound}
\|\nabla u\|_{L^\infty(\Omega_{R_{1}})} \leq \frac{C}{|\Sigma'|+\va\rho_{n}^{-1}(\varepsilon)} \|\varphi\|_{C^{2, \alp}(\partial D)}\leq\,\frac{C}{|\Sigma'|}\|\varphi\|_{C^{2, \alp}(\partial D)},
\end{equation}
and
$$\|\nabla u\|_{L^{\infty}(\Omega\setminus\Omega_{R_{1}})}\leq\,C\|\varphi\|_{C^{2, \alp}(\ptl D)},$$
where $C$ is independent of $\va$, $|\Sigma'|$ denotes the area of $\Sigma'$,  and
\begin{equation}\label{def_rhon}
\rho_{n}(\varepsilon)=
\begin{cases}
\sqrt{\varepsilon}\quad& n=2,\\
\frac{1}{|\mathrm{ln}\va|}& n=3,\\
1& n\geq4.
\end{cases}
\end{equation}
\end{theorem}
\begin{remark}
The proof of Theorem \ref{thm1} actually gives a stronger pointwise upper bound estimate:
\begin{align}\label{upperbound1'}
|\nabla{u}(x)|\leq \frac{C\varepsilon}{|\Sigma'|+\va\rho_{n}^{-1}(\varepsilon)}\cdot\frac{1}{\varepsilon+\dist^{2}(x, \Sigma)}\cdot\|\varphi\|_{C^{2, \alp}(\partial D)},\quad\mbox{for}~x\in\Omega_{R_{1}}.
\end{align}
where $\Sigma:=\Sigma'\times(-\frac{\va}{2},\frac{\va}{2})$.
\end{remark}

\begin{remark}\label{rmk1}
If $\Sigma'=\{0'\}$, then $|\Sigma^\prime|=0$. From \eqref{upperbound1'}, we can see that the upper bound estimate actually is
$$|\nabla{u}(0',x_{n})|\leq \frac{C\rho_{n}(\va)}{\va}\cdot\|\varphi\|_{C^{2, \alp}(\partial D)},\quad -\frac{\va}{2}<x_n<\frac{\va}{2}.
$$
This is consistent with the known results, see \cite{bly1,akl,aklll,ly} for instance. While, if $|\Sigma'|>0$, then $|\nabla u|$ is bounded in $\Omega$  from \eqref{upperbound1'}, which implies that no blow-up occurs.
\end{remark}

In contrast with the blow-up result of Theorem \ref{thm-interior}, Theorem \ref{thm1} shows the boundedness of $|\nabla u|$ whenever $|\Sigma'|>0$. In order to further reveal such blow-up mechanics, we consider the intermediate cases that the relative convexity between $D_1$ and $D_2$ is of order $2\leq\,m<\infty$. Namely, we assume that, besides of \eqref{h1h2'}--\eqref{h1h1}, 
\begin{equation}\label{h1h_convex1-1}
|\nabla h_i(x^\prime)|\leq C|x^\prime|^{m-1}, \quad |\nabla^2 h_i(x^\prime)|\leq C|x^\prime|^{m-2},\quad \text{for} ~|x^\prime|\leq 2R_1,~i=1,\, 2,
\end{equation}
and
\begin{equation}\label{h1h_convex1}
\lam_0|x'|^{m}\leq h_{1}(x')-h_{2}(x')\leq \lam_1|x'|^{m}, \quad \text{for} ~|x'|<2R_{1},
\end{equation}
for $0<\lambda_{0}<\lambda_{1}$. By a slight modification of the proof of Theorem \ref{thm1}, we have the following upper bound estimates with different blow-up rates, which tend to $O(1)$ as $m\to\infty$.

\begin{theorem}\label{rem-m-geq2}
Let $D_{1}, D_{2}\subset{D}\subset\mathbb{R}^{n}$ ($n\geq2$) be two bounded $C^{2,\alp}$ subdomains with $\Sigma'=\{0'\}$. Suppose \eqref{h1h2'}--\eqref{h1h1},  \eqref{h1h_convex1-1} and \eqref{h1h_convex1} hold. Let $u\in{H}^{1}(D)\cap{C}^{1}(\overline{\Omega})$
be the solution of \eqref{equ-infty-interior} and $\varphi \in C^{2, \alp}(\ptl D)$.  Then for small $\va>0$, 
\begin{align}\label{upperbound-m-infty}
|\nabla{u}(x)|\leq \frac{C \rho_{n,\,m}(\varepsilon)}{\varepsilon+|x'|^{ m}}\cdot\|\varphi\|_{C^{2, \alp}(\partial D)},\quad\mbox{for}~x\in\Omega_{R_{1}},
\end{align}
and
$$\|\nabla u\|_{L^{\infty}(\Omega\setminus\Omega_{R_{1}})}\leq\,C\|\varphi\|_{C^{2, \alp}(\ptl D)},$$
where
\begin{align*}
\rho_{n,\,m}(\varepsilon)=\begin{cases}\varepsilon^{1-\frac{n-1}{m}}& m>n-1,\\
\frac{1}{|\ln\varepsilon|}& m=n-1,\\
1 & m<n-1.
\end{cases}
\end{align*}
\end{theorem}

    The rest of this paper is organized as follows. In Section 2, we first decompose the solution $u=C_{1}v_{1}+C_{2}v_{2}+v_{0}$, where $C_{i}$ are from the free constant boundary conditions $u=C_{i}$,  and $v_{i}$ are solutions of boundary value problems with given Dirichlet data. Then we reduce the proof of Theorem \ref{thm-interior} and \ref{thm1} to the estimates for $\nabla v_{i}$ and the estimates for $C_{i}$. The key estimate is the pointwise upper and lower bounds for $|\nabla v_{1}|$ in the narrow region $\Omega_{R_1}$, see Propositions \ref{prop-interior} and \ref{prop1}, which will also be used to estimate $|C_{1}-C_{2}|$. When $\partial D_{1}$ and $\partial D_{2}$ are of $C^{1,\alpha}$, in order to prove Proposition \ref{prop-interior}, we need to adapt the classical $C^{1, \alp}$ estimates \cite{gm} to our setting with partially zero boundary condition, see Theorem \ref{lem-global-C1alp-estimates}. It can be regarded as the analogue of theorem 9.13 in \cite{gt}.  We give its proof in the Appendix. In Section \ref{sec3}, we use Theorem \ref{lem-global-C1alp-estimates} to prove Proposition \ref{prop-interior}. In Section \ref{sec4}, we use the iteration technique with respect to the energy, developed in \cite{bll}, and $W^{2,p}$ estimates to prove Proposition \ref{prop1} and Proposition \ref{lemma-nabla-v1+v2}.

\vspace{0.5cm}

\section{Outline of the proof for two main results}

In this section, we shall give the main ingredients to prove Theorem \ref{thm-interior} and \ref{thm1} and outline the proofs. 

We first use the following decomposition as in \cite{bly1}: 
\begin{equation}\label{decomposition1_u}
u(x)=C_{1}v_{1}(x)+C_{2}v_{2}(x)+v_{0}(x),\qquad\qquad~x\in\,\Omega ,
\end{equation}
where $C_i := C_i (\va)$ is the free boundary value of $u$ on $\ptl D_i$, $i=1, \, 2$, to be determined by $u$.  Meanwhile,  $v_{j}\in{C}^{2}(\Omega)~(j=0,1,2)$, respectively, satisfies
\begin{equation}\label{equ_v1}
\begin{cases}
\Delta{v}_{1}=0& \mathrm{in}~\Omega,\\
v_{1}=1& \mathrm{on}~\partial{D}_{1},\\
v_{1}=0& \mathrm{on}~\partial{D_{2}}\cup\partial{D},
\end{cases}
\quad
\begin{cases}
\Delta{v}_{2}=0& \mathrm{in}~\Omega,\\
v_{2}=1& \mathrm{on}~\partial{D}_{2},\\
v_{2}=0& \mathrm{on}~\partial{D_{1}}\cup\partial{D},
\end{cases}
\end{equation}
and
\begin{equation}\label{equ_v3}
\begin{cases}
\Delta{v}_{0}=0& \mathrm{in}~\Omega,\\
v_{0}=0& \mathrm{on}~\partial{D}_{1}\cup\partial{D_{2}},\\
v_{0}=\varphi& \mathrm{on}~\partial{D}.
\end{cases}
\end{equation}
From \eqref{decomposition1_u}, one has
\begin{align}\label{decomposition_u2}
\nabla{u}(x)=&C_{1}\nabla{v}_{1}(x)+C_{2}\nabla{v}_{2}(x)+\nabla{v}_{0}(x) \notag \\
=&(C_{1}-C_{2})\nabla{v}_{1}(x)+C_{2}\nabla({v}_{1}+{v}_{2})(x)+\nabla{v}_{0}(x),\qquad~x\in\,\Omega.
\end{align}
Thus, in order to prove \eqref{upper-bounds} and \eqref{result-bound}, it suffices to estimate $|\nabla v_1|$, $|\nabla (v_1+v_2)|$, $|\nabla v_0|$, $|C_1-C_2|$ and $|C_i|$ ($i=1, \, 2$), respectively.

To estimate $|\nabla v_1|$, we introduce an auxiliary function $\bar{u}_{1}\in{C}^{1, \alp}(\mathbb{R}^{n})$ such that $\bar{u}_{1}=1$ on $\partial{D}_{1}$, $\bar{u}_{1}=0$ on $\partial{D}_{2}\cup\partial D$,
\begin{align}\label{ubar}
\bar{u}_{1}(x', x_n)
=\frac{x_{n}-h_{2}(x')+\va/2}{\varepsilon+h_{1}(x')-h_{2}(x')},\ \ \qquad~(x', x_n)\in\,\Omega_{R_{1}},
\end{align}
and
\begin{equation}\label{ubar-out}
\|\bar{u}_{1}\|_{C^{1,\alp}(\Omega\setminus \Omega_{R_{1}})}\leq\,C.
\end{equation}
Here and throughout this paper, unless otherwise stated, $C$ denotes a constant, whose value may vary from line to line, depending only on $n,\lambda_{0},\lambda_{1}, \kappa_{0},\ \kappa_1$, $ \kappa_2,\ \kappa_3$, $ R_{1}$ and an upper bound of the $C^{2,\alpha}$ (or $C^{1,\alpha}$) norms of $\partial{D}_{1}$ and $\partial{D}_2$, but not on $\varepsilon$. Also, we call a constant having such dependence a {\it universal constant}.

By using \eqref{h1h1}, \eqref{h1h_convex2} and denoting $\ptl_{i}:=\ptl/\ptl x_i$,  
a direct calculation yields that
\begin{equation}\label{nablau_bar-interior}
\left|\partial_{i}\bar{u}_{1}(x)\right|\leq\frac{C|x'|^{\alp}}{\varepsilon+|x'|^{1+\alp}},\quad\,i=1,\cdots,n-1,\quad \partial_{n}\bar{u}_{1}=\frac{1}{\delta(x^\prime)},  ~\mbox{for}~ x\in \Omega_{R_1},
\end{equation}
where
\begin{equation}\label{delta_x}
\delta(x'):=\varepsilon+h_1(x')-h_2(x').
\end{equation}

Set
$$w:=v_{1}-\bar{u}_{1}.$$
Then
\begin{prop}\label{prop-interior}
Under the hypotheses of Theorem \ref{thm-interior}. Let $v_1$  $\in H^1(\Omega )$ be the weak solution of \eqref{equ_v1}. Then for small $\va >0$, 
\begin{equation}\label{nabla-w-i0}
|\nabla w(x^\prime, x_n)|\leq \frac{C}{(\va+|x^\prime|^{1+\alp})^{\frac{1}{1+\alp}}},\quad  (x^\prime,x_n)\in \Omega_{ R_1}, 
\end{equation}
where $C$ is independent of $\va$. Consequently,
\begin{equation}\label{nabla-v-i0}
\frac{1}{C(\va +|x^\prime|^{1+\alp})}\leq|\nabla v_{1}(x^\prime, x_n)|\leq \frac{C}{\va +|x^\prime|^{1+\alp}},\quad  (x^\prime,x_n)\in \Omega_{ R_1}. 
\end{equation}
\end{prop}

Because $h_{1}$ and $h_{2}$ here are only of $C^{1,\alpha}$, now $\bar{u}_{1}$ is not twice continuously differentiable. Thus, we do not have $-\Delta w=\Delta \bar{u}_{1}$ any more, and can not immediately apply $W^{2,\, p }$ estimates to obtain $|\nabla w|\leq\,C$, like in \cite{LX}. We now write the right hand side in divergence form
$$-\Delta w=\mathrm{div} (\nabla\bar{u}_{1}).$$
We turn to the $C^{1, \alp}$ estimates for elliptic equations to prove Proposition \ref{prop-interior}.  With the aid of  De Giorgi-Nash estimates, we adapt the classical $C^{1, \alp}$ estimates \cite{gm} to our setting with partially zero boundary condition, which can be regarded as the analogue of theorem 9.13 in \cite{gt}. For readers' convenience, we give its proof in the Appendix.

\begin{theorem}\label{lem-global-C1alp-estimates}
Let $Q$ be a bounded domain in $\R^n$, $n\geq 2$, with a $C^{1,\alp}$ boundary portion $\Gamma \subset \ptl Q$.  Let $\widetilde{w} \in{H}^{1}(Q)\cap{C}^{1}(Q \cup \Gamma)$ be the solution of
\begin{equation}\label{equ-w-divf-q1}
\left\{ \begin{aligned}
-\Delta \widetilde{w}&=\Div \tilde{\mathbf{f}}\quad &\text{in}\quad Q, \\
\widetilde{w}&=0 \quad &\text{on} \quad \Gamma,
\end{aligned}\right.
\end{equation}
where $\tilde{\mathbf{f}}\in C^{\alp}(Q, \R^n)$, $0<\alp<1$. Then for any domain $Q^\prime \subset\subset Q \cup \Gamma$,
\begin{equation}\label{ineq-global-C1alp-estimates}
\|   \widetilde{w} \|_{C^{1, \, \alp}(Q^\prime)} \leq C\left( \|  \widetilde{w} \|_{L^\infty( Q)}+[\tilde{\mathbf{f}}]_{\alp,\, Q}\right),
\end{equation}
where $C=C(n, \alp, Q^\prime, Q)$.
\end{theorem}
Here, the H\"{o}lder semi-norm of $\tilde{\mathbf{f}}=(\widetilde{f}_1, \, \widetilde{f}_2, \, ...\, , \,\widetilde{f}_n)$ is defined as follows:
\begin{equation}\label{def-nablaU-alp}
[\tilde{\mathbf{f}}]_{\alp,\, Q}:=\max\limits_{1\leq i \leq n}[\widetilde{f}_i]_{\alp, \, Q}\quad \text{and} \quad [\widetilde{f}_i]_{\alp, \, Q}=\sup_{x, \, y\in Q}\frac{|\widetilde{f}_i(x)-\widetilde{f}_i(y)|}{|x-y|^\alp}.
\end{equation}

In order to prove Theorem \ref{thm-interior}, we also need
\begin{lemma}\label{lemma-m-1-to-2} 
Under the hypotheses of Theorem \ref{thm-interior}, let $v_{i}\in{H}^1(\Omega)$ $(i=0,1,2)$ be the
weak solutions of \eqref{equ_v1} and \eqref{equ_v3} respectively. Then
\begin{align}
&|\nabla(v_{1}+v_{2})(x)|\leq\,C,\qquad\,x\in\Omega; \label{v1+v2_bounded1}\\
&|\nabla{v}_{0}(x)|\leq C\|\varphi\|_{L^{\infty}(\partial D)},\quad\,x\in\Omega; \label{nabla_v3}
\end{align}
and
\begin{align}
& |C_i|\leq C, \quad \text{for} \quad i=1, \, 2,  \label{esti-C12}\\
&  |C_1-C_2|\leq C\rho_{n,\,\alp}(\va)\cdot\|\varphi\|_{C^{1, \, \alp}(\ptl D)},  \label{esti-c1-c2}
\end{align}
where $C$ is independent of $\va$.
\end{lemma}

We are in position to prove Theorem \ref{thm-interior} by using Proposition \ref{prop-interior} and Lemma \ref{lemma-m-1-to-2}.

\begin{proof}[Proof of Theorem \ref{thm-interior}]
It follows from \eqref{decomposition_u2}, \eqref{nabla-v-i0} and Lemma \ref{lemma-m-1-to-2} that for $x\in\Omega_{R_{1}}$,
\begin{align*}
|\nabla{u}(x)|\leq&|C_{1}-C_{2}|\cdot|\nabla{v}_{1}(x)|+C+C\|\varphi\|_{C^{1, \alp}(\partial D)}\\
\leq&  \frac{C\rho_{n,\,\alp}(\va)}{\varepsilon+|x^\prime|^{1+\alp}}\cdot\|\varphi\|_{C^{1, \alp}(\partial D)}.
\end{align*}
For the rest part, it is easy to see from the standard elliptic theories and Lemma \ref{lemma-m-1-to-2} that
\begin{equation*}
\|\nabla{u}\|_{L^{\infty}(\Omega\setminus\Omega_{R_{1}})}\leq C \|\varphi\|_{C^{1, \, \alp}(\ptl D)}.
\end{equation*}
The proof of Theorem \ref{thm-interior} is completed. 
\end{proof}

When $\partial{D}_{1}$ and $\partial{D}_{2}$ are of $C^{2,\alpha}$, the auxiliary function is the same as before, still denoting $\bar{u}_1$, but good enough to take twice derivative. From the assumptions on $h_{1}$ and $h_{2}$, \eqref{h1h2}--\eqref{h1h3}, a direct calculation gives
\begin{equation*}
\partial_{i}\bar{u}_{1}(x)=0,~~\,i=1,\cdots,n-1,\qquad
\partial_{n}\bar{u}_{1}(x)=\frac{1}{\varepsilon},\qquad\qquad~x\in\Sigma,
\end{equation*}
and
\begin{equation}\label{nablau_bar2}
\left|\partial_{i}\bar{u}_{1}(x)\right|\leq\frac{Cd_{\Sigma^\prime}(x')}{\varepsilon+d^{2}_{\Sigma^\prime}(x')},\qquad
\partial_{n}\bar{u}_{1}(x)=\frac{1}{\delta(x')},\qquad\qquad~x\in\Omega_{R_{1}}\setminus\Sigma,
\end{equation}
where $  d_{\Sigma^\prime}(x^\prime):=\text{dist}(x^\prime, \, \Sigma^\prime)$ and $\delta(x')=\va+h_1(x')-h_2(x')$. 

\begin{prop}\label{prop1}
Under the hypotheses of Theorem \ref{thm1}, let $v_{1} \in{H}^1(\Omega)$ be the
weak solution of \eqref{equ_v1}. Then  
\begin{equation}\label{nabla_w_i0}
\|\nabla(v_{1}-\bar{u}_{1})\|_{L^{\infty}(\Omega_{R_{1}})}\leq\,C,
\end{equation}
where $C$ is independent of $\va$. As a consequence, for small $\va$, we have
\begin{equation}\label{v1-bounded1}
\frac{1}{C(\varepsilon+d^{2}_{\Sigma^\prime}(x'))}\leq|\nabla v_{1}(x)|\leq\,\frac{C}{\varepsilon+d^{2}_{\Sigma^\prime}(x')},\quad x\in\Omega_{R_{1}}.
\end{equation}
\end{prop}

Instead of Lemma \ref{lemma-m-1-to-2}, we have
\begin{lemma}\label{lemma-nabla-v1+v2}
 Under the hypotheses of Theorem \ref{thm1}, let $v_{i}\in{H}^1(\Omega)$ $(i=0,1,2)$ be the
 weak solutions of \eqref{equ_v1}--\eqref{equ_v3}, respectively. Then \eqref{v1+v2_bounded1}, \eqref{nabla_v3} and \eqref{esti-C12} hold, and instead of \eqref{esti-c1-c2}, we have
 \begin{equation}\label{esti-c1-and-c2}
  |C_1-C_2| \leq \frac{C\va}{|\Sigma^\prime|+\rho_n^{-1}(\va)\va}\cdot \|\varphi\|_{L^\infty(\ptl D)}, 
 \end{equation}
 where $\rho_{n}(\va)$ is defined in \eqref{def_rhon} and $C$ is independent of $\va$.
 \end{lemma}
Combining with these estimates above, the proof of Theorem \ref{thm1} follows that of Theorem \ref{thm-interior}.

\vspace{0.5cm}

\section{Estimates for $C^{1,\alp}$ inclusions and the proof of Proposition \ref{prop-interior} }\label{sec3}

This section is devoted to the proof of Proposition \ref{prop-interior}. Because $\ptl D_1$ and $\ptl D_2$ are only $C^{1, \alp}$, we adapt the iteration technique developed in \cite{bll} to allow us to apply Theorem \ref{lem-global-C1alp-estimates}.  In this end, we define
 \begin{equation*}
\widehat{\Omega}_{s}(z'):=\left\{x\in \mathbb{R}^{n}:-\frac{\va}{2}+h_{2}(x')<x_{n}<\frac{\va}{2}+h_{1}(x'),~|x'-z'|<{s}~\right\},
\end{equation*}
for $0<s<\frac{1}{2\kappa_{1}}\delta(z')^{1/(1+\alpha)}\leq R_1$, $\kappa_{1}$ is defined in \eqref{h1h_convex2}. We first calculate the semi-norm
\begin{equation}\label{ineq-semi-holder-norm}
[\nabla \bar{u}_1]_{\alp, \, \widehat{\Omega}_{s}(z^\prime) }\leq C\delta(z^\prime)^{-\frac{2+\alp}{1+\alp}}s^{1-\alp}+C\delta(z^\prime)^{-\frac{1+\alp+\alp^2}{1+\alp}}.
\end{equation}
where $\delta(z')=\va + h_1(z^\prime)-h_2(z^\prime)$ and $s\leq C\delta(z^\prime)$.

Indeed, we first note that for any $(x^\prime, x_n)\in \widehat{\Omega}_{s}(z^\prime)$, $s\leq\delta(z')$,
\begin{equation}\label{ineq-x-prime}
|x^\prime|\leq |x^\prime-z^\prime|+|z^\prime|<s+|z^\prime|\leq C\delta(z')^{1/(1+\alpha)}.
\end{equation}
This together with mean value theorem  and \eqref{h1h_convex2} implies that for any $x, \bar{x}\in \widehat{\Omega}_{s}(z^\prime)$ with $x^\prime\neq \bar{x}^\prime$, 
 \begin{equation}\label{ineq-meanvalue}
|h_i(x^\prime)-h_i(\bar{x}^\prime)|=|\nabla h_{i}(x'_{\theta_{i}})\|x^\prime-\bar{x}^\prime|\leq C\delta(z')^{\frac{\alp}{1+\alp}}|x^\prime-\bar{x}^\prime|, \quad i=1,\,2,
 \end{equation}
 and
 \begin{equation}\label{ineq-meanvalue-1}
\va + (h_1-h_2)(\bar{x}^\prime)\geq\delta(z')-C\delta(z')^{\frac{\alp}{1+\alp}}s\geq\frac{1}{2}\delta(z'),\quad\va + (h_1-h_2)(x^\prime)\geq\frac{1}{2}\delta(z').
 \end{equation}
Then, for
\begin{equation*}
\ptl_n\bar{u}_1(x)=\frac{1}{\va+h_1(x^\prime)-h_2(x^\prime)},
\end{equation*}
we have
\begin{align}\label{ineq-xn-alp}
\frac{|\partial_{n}\bar{u}_{1}(x)-\partial_{n}\bar{u}_{1}(\bar{x})|}{|x-\bar{x}|^\alp}
\leq \frac{C\delta(z^\prime)^{\frac{\alp}{1+\alp}} s^{1-\alp}}{\delta(z^\prime)^2}\leq\,C\delta(z^\prime)^{-\frac{2+\alp}{1+\alp}}s^{1-\alp}.
\end{align}
While, for $i= 1, \, 2, \, \cdots, \, n-1$,
\begin{align*}
\ptl_i\bar{u}_1(x)& =\frac{-\partial_{i}h_2(x^\prime)}{\delta(x^\prime)}+\frac{\big(x_n-h_2(x^\prime)+\va/2 \big)\big(\partial_{i}h_2(x^\prime)-\partial_{i}h_1(x^\prime)\big)}{\delta^2(x^\prime)}\\
:&=\Phi_1(x)+\Phi_2(x),
\end{align*}
 we have
 \begin{align*}
 \frac{|\ptl_i\bar{u}_1(x)-\ptl_i \bar{u}_1(\bar{x})|}{|x-\bar{x}|^\alp}\leq \frac{|\Phi_1(x)-\Phi_1(\bar{x})|}{|x-\bar{x}|^\alp}+\frac{|\Phi_2(x)-\Phi_2(\bar{x})|}{|x-\bar{x}|^\alp}:=\mathrm{I}_1+\mathrm{I}_2.
 \end{align*}
By virtue of \eqref{h1h_convex2} and \eqref{ineq-x-prime}--\eqref{ineq-meanvalue-1}, a direct calculation yields
\begin{align*}
\mathrm{I}_1\leq\frac{C}{\delta(z^\prime)}+
\frac{\delta(z^\prime)^{\frac{\alp}{1+\alp}} s^{1-\alp}}{\delta(z^\prime)^2}\leq
 C\delta(z^\prime)^{-1}+C\delta(z^\prime)^{-\frac{2+\alp}{1+\alp}}s^{1-\alp},
\end{align*}
and
\begin{equation*}
\mathrm{I}_2\leq C\delta(z^\prime)^{-1-\frac{1}{1+\alp}}s^{1-\alp}+C\delta(z^\prime)^{-\alp-\frac{1}{1+\alp}}.
\end{equation*}
Noting that $\alp+\frac{1}{1+\alp}>1$, we have 
\begin{equation}\label{ineq-xi-alp}
\frac{|\partial_{i}\bar{u}_{1}(x)-\partial_{i}\bar{u}_{1}(\bar{x})|}{|x-\bar{x}|^\alp}  \leq \mathrm{I}_1+\mathrm{I}_2\leq C\delta(z^\prime)^{-\frac{2+\alp}{1+\alp}}s^{1-\alp}+C\delta(z^\prime)^{-\alp-\frac{1}{1+\alp}}.
\end{equation}
Thus, \eqref{ineq-semi-holder-norm} immediately follows from \eqref{ineq-xn-alp} and \eqref{ineq-xi-alp}.

\subsection{Proof of Proposition \ref{prop-interior}} 

\begin{proof}[Proof of Proposition \ref{prop-interior}]
Recall $w$ satisfies
\begin{equation}\label{w20'}
\left\{
\begin{aligned}
-\Delta{w}&=\mbox{div}(\nabla\bar{u}_{1})\quad&\mbox{in}~\Omega,\quad\\
w&=0\quad&\mbox{on}~\partial \Omega.
\end{aligned}\right.
\end{equation}
Since 
\begin{equation}\label{nabla-baru-out}
|\bar{u}_1|+|\nabla \bar{u}_1|\leq C\quad \text{in}~\Omega\setminus \Omega_{R_1/2},
\end{equation}
it follows from the standard elliptic theories that
\begin{equation}\label{nabla-w-out}
|w|+\left|\nabla{w}\right|\leq\,C,
\quad\mbox{in}~~ \Omega\setminus\Omega_{R_{1}}.
\end{equation}
Thus, it is clear from \eqref{ubar-out} that
\begin{equation}\label{ineq-nabla-v-out}
|\nabla v_1(x)|\leq C, \qquad x\in \Omega \setminus \Omega_{R_1}.
\end{equation}
In order to estimate $\|\nabla v_1\|_{L^\infty (\Omega_{R_1})}$,  we divide the proof into three steps.

{\bf STEP 1. The boundedness of the total energy:} 
\begin{equation}\label{energy-nabla-w11}
\int_{\Omega}|\nabla {w}|^2\ dx\leq C.
\end{equation}
In fact, noting that $\ptl_{nn} \bar{u}_1=0$ in $\Omega_{R_1}$, we multiply \eqref{w20'} by $w$, make use of integration by parts and Young's inequality, to obtain
\begin{align*}
\int_{\Omega}|\nabla w|^2dx&=\int_{\Omega\setminus\Omega_{R_1}}\Div (\nabla \bar{u}_1)\, w\,dx+\int_{\Omega_{R_1}}\Div (\nabla\bar{u}_1) \,w\,dx\\
&\leq \int_{\Omega\setminus\Omega_{R_1}}|\nabla\bar{u}_1|\,|\nabla w|dx+\int_{\ptl \Omega_{R_1}\setminus \ptl \Omega}|\nabla \bar{u}_1|\,|w|ds+\frac{1}{2}\int_{\Omega_{R_1}}|\nabla w|^2\,dx\\
&\quad+\frac{1}{2}\int_{\Omega_{R_1}}\sum_{i=1}^{n-1}|\ptl_i\bar{u}_1|^2\,dx+\int_{\ptl\Omega_{R_1}\setminus\ptl \Omega}\sum_{i=1}^{n-1}|\ptl_i\bar{u}_1|\,|w|\,ds.
\end{align*}
Then, using \eqref{nabla-baru-out} and \eqref{nabla-w-out}, one has
\begin{equation*}
\int_{\Omega}|\nabla {w}|^2\ dx\leq   \int_{\Omega_{R_1}} \sum_{i=1}^{n-1} |\ptl_i \bar{u}_1|^2 \ dx+C.
\end{equation*}
For the first term on the right hand side, by using \eqref{nablau_bar-interior}, we have 
\begin{align*}
\int_{\Omega_{R_1}} \sum_{i=1}^{n-1} |\ptl_i \bar{u}_1|^2 \ dx &\leq \int_{|x^\prime|\leq R_1}\int_{-\frac{\va}{2}+h_2(x^\prime)}^{\frac{\va}{2}+h_1(x^\prime)}\frac{C|x^\prime|^{2\alp}}{(\va+|x^\prime|^{1+\alp})^2} \ dx_n dx^\prime\\
&\leq C\int_{0}^{R_1}\frac{r^{n+2\alp-2}}{\va+r^{1+\alp}} \ dr\leq C\int_{0}^{R_1}r^{n+\alp-3} dr \leq C.
\end{align*}
So that \eqref{energy-nabla-w11} is proved.

{\bf STEP 2. The local energy estimates:}
\begin{equation}\label{energy_w11_inomega_z1}
\int_{\widehat{\Omega}_{\delta(z')}(z')}\left|\nabla{w}\right|^{2}\ dx\leq
C\delta(z')^{n-\frac{2}{1+\alp}},
\end{equation}
where $\delta(z')=\va+h_1(z')-h_2(z')$.

Indeed, from \eqref{w20'}, we see that $w$ also satisfies 
\begin{equation}\label{equ-nabla-bar-u}
-\Delta w= \Div (\nabla \bar{u}_1-\mathfrak{a})\quad\mbox{in}~\Omega,
\end{equation} 
for any constant vector $\mathfrak{a}\in \R^n$. For $0<t<s<R_{1}$, let $\eta$ be a cutoff function satisfying 
\begin{align}\label{def-eta-x-prime}
\eta(x^\prime)=\begin{cases}
1 & \text{if} ~ |x'-z'|<t, \\
0 & \text{if} ~ |x'-z'|>s,\\
\end{cases}
\quad\text{and}\quad |\nabla_{x'}\eta(x')|\leq\frac{2}{s-t}.
\end{align}
Multiplying  \eqref{equ-nabla-bar-u} by $\eta^{2}w$ and using integration by parts, one has
\begin{equation}\label{ineq-iteration-w1}
\begin{aligned}
\int_{\widehat{\Omega}_{t}(z')}|\nabla w|^{2}\ dx \leq\,\frac{C}{(s-t)^{2}}\int_{\widehat{\Omega}_{s}(z')}|w|^{2}\ dx+ C\int_{\widehat{\Omega}_{s}(z')}|\nabla \bar{u}_1-\mathfrak{a}|^2 \ dx,
\end{aligned}
\end{equation}
where we take
\[\mathfrak{a}=(\nabla \bar{u}_1)_{\widehat{\Omega}_{s}(z')}:=\frac{1}{|\widehat{\Omega}_{s}(z')|}\int_{\widehat{\Omega}_{s}(z')}\nabla \bar{u}_1(y) \ dy.\]

{\bf Case 1.} For $|z'|\leq \va^{\frac{1}{1+\alp}}$, $0<s<\va^{\frac{1}{1+\alp}}$, then $\va\leq\delta(z')\leq\,C\varepsilon$. By a direct calculation, we have
\begin{align}\label{energy_nabla_w11_square_in}
\int_{\widehat{\Omega}_{s}(z')}|w|^{2}&=\int_{|x'-z'|<s}\int_{-\va/2+h_2}^{\va/2+h_1}\left(\int_{-\frac{\va}{2}+h_2}^{x_{n}}\partial_{ n}w \ dx_{n}\right)^{2}\ dx_{n}dx'\nonumber\\
&\leq C\varepsilon^{2}\int_{\widehat{\Omega}_{s}(z')}|\nabla{w}|^{2} \ dx,
\end{align}
and by the definition of semi-norm $[\cdot]_{\alp, \, \widehat{\Omega}_{s}(z')}$,
\begin{align*}
|\nabla \bar{u}_1-(\nabla \bar{u}_1)_{\widehat{\Omega}_{s}(z')}|&\leq \frac{1}{|\widehat{\Omega}_{s}(z')|}\int_{\widehat{\Omega}_{s}(z')}|\nabla  \bar{u}_1(x)-\nabla \bar{u}_1(y)| \ dy  \\
&\leq \frac{[\nabla \bar{u}_1]_{\alp, \, \widehat{\Omega}_{s}(z')}}{|\widehat{\Omega}_{s}(z')|}\int_{\widehat{\Omega}_{s}(z')} |x-y|^{\alp} \ dy \\
& \leq C [\nabla \bar{u}_1]_{\alp, \, \widehat{\Omega}_{s}(z')} (s^{\alp}+\delta(z')^{\alp}).
\end{align*}
Using \eqref{ineq-semi-holder-norm}, we calculate further
\begin{align}\label{ineq-s-va}
\int_{\widehat{\Omega}_{s}(z')} |\nabla \bar{u}_1-(\nabla \bar{u}_1)_{\widehat{\Omega}_{s}(z')}|^2 \ dx&\leq \frac{Cs^{n+1}}{\va^{1+\frac{2}{1+\alp}}}
+ \frac{Cs^{n-1}}{\va^{\frac{2}{1+\alp}-1}}
+ \frac{Cs^{n+2\alp-1}}{\va^{2\alp+\frac{2}{1+\alp}-1}}
\notag \\ &\quad +\frac{Cs^{n+1-2\alp}}{\va^{1+\frac{2}{1+\alp}-2\alp}}:=G(s).
\end{align}
It follows from \eqref{ineq-iteration-w1}, 
\eqref{energy_nabla_w11_square_in} and \eqref{ineq-s-va}  that 
\begin{equation}\label{ineq-F111_in}
F(t)\leq\,\left(\frac{c_{1}\varepsilon}{s-t}\right)^{2}F(s) +CG(s),\quad \forall~ 0<t<s<\sqrt[m]{\varepsilon},
\end{equation}
here $c_1$ is a fixed constant, and
\begin{equation}\label{denotation}
F(t):=\int_{\widehat{\Omega}_{t}(z')}|\nabla{w}|^{2} \ dx.
\end{equation}
Let $k=\left[\frac{1}{4c_{1}\varepsilon^{\frac{\alp}{1+\alp}}}\right]$ and $t_{i}=\delta+2c_{1}i\varepsilon$, $i=0,1,2,\cdots,k$. It is easy to see from \eqref{ineq-s-va} that
\begin{align*}
G(t_{i+1})\leq C(i+1)^{n+1}\va^{n-\frac{2}{1+\alp}}.
\end{align*}
Taking $s=t_{i+1}$ and $t=t_{i}$ in \eqref{ineq-F111_in}, we have the following iteration formula
$$F(t_{i})\leq\,\frac{1}{4}F(t_{i+1}) +C(i+1)^{n+1}\va^{n-\frac{2}{1+\alp}}.$$
After $k$ iterations, and by virtue of \eqref{energy-nabla-w11}, we have
\[F(t_{0})\leq (\frac{1}{4})^{k}F(t_{k})+C\va^{n-\frac{2}{1+\alp}}\sum_{i=0}^{k-1}(\frac{1}{4})^i(i+1)^{n+1} \leq C\va^{n-\frac{2}{1+\alp}}.\]
This is \eqref{energy_w11_inomega_z1} with $\delta(z')\leq\,C\va$. 

{\bf Case 2.} For $\va^{\frac{1}{1+\alp}}\leq|z'|\leq\,R_{1}$, $0<s<|z'|$, then $|z'|^{1+\alpha}\leq\delta(z')\leq\,C|z'|^{1+\alpha}$. The estimates \eqref{energy_nabla_w11_square_in} and \eqref{ineq-s-va}  become, respectively,
\begin{align}\label{energy_nabla_w11_square}
\int_{\widehat{\Omega}_{s}(z')}|w|^{2}\ dx
\leq&\,C|z'|^{2(1+\alp)}\int_{\widehat{\Omega}_{s}(z')}|\nabla{w}|^{2}\ dx, \quad\mbox{if}~\,0<s<\frac{2}{3}|z'|,
\end{align}
and
\begin{align}\label{ineq-H-z}
\int_{\widehat{\Omega}_{s}(z')} |\nabla \bar{u}_1-(\nabla \bar{u}_1)_{\widehat{\Omega}_{s}(z')}|^2 \ dx&\leq \frac{Cs^{n+1}}{|z^\prime|^{\alp+3}}
+ \frac{Cs^{n-1}}{|z^\prime|^{1-\alp}}
+ \frac{Cs^{n+2\alp-1}}{|z^\prime|^{2\alp^2+\alp+1}}
\notag \\ &\quad +\frac{Cs^{n+1-2\alp}}{|z^\prime|^{3-\alp-2\alp^2}}:=H(s).
\end{align}
In view of \eqref{ineq-iteration-w1}, and \eqref{energy_nabla_w11_square}, estimate \eqref{ineq-F111_in} becomes, 
\begin{equation}\label{ineq-ft-fs}
F(t)\leq\,\left(\frac{c_{2}|z'|^{1+\alp}}{s-t}\right)^{2}F(s)+CH(s),
\quad\forall~0<t<s<\frac{2}{3}|z'|,
\end{equation}
where $c_2$ is another fixed constant.
Let $k=\left[\frac{1}{4c_{2}|z'|^{\alp}}\right]$ and $t_{i}=\delta+2c_{2}i\,|z'|^{1+\alp}$, $i=0,1,2,\cdots,k$. From \eqref{ineq-H-z}, one has
\begin{align*}
H(t_{i+1}) \leq C(i+1)^{n+1}|z^\prime|^{(1+\alp)(n-\frac{2}{1+\alp})}.
\end{align*}
Then, taking $s=t_{i+1}$ and $t=t_{i}$ in \eqref{ineq-ft-fs}, the iteration formula is
$$F(t_{i})\leq\,\frac{1}{4}F(t_{i+1}) +C(i+1)^{n+1}|z^\prime|^{(1+\alp)(n-\frac{2}{1+\alp})}.$$
After $k$ iterations, and using \eqref{energy-nabla-w11} again, 
\begin{equation*}
F(t_{0}) \leq (\frac{1}{4})^{k}F(t_{k})+C|z^\prime|^{(1+\alp)(n-\frac{2}{1+\alp})}\sum_{i=0}^{k-1}(\frac{1}{4})^i(i+1)^{n+1}\leq C|z'|^{(1+\alp)(n-\frac{2}{1+\alp})}.
\end{equation*}
Thus, \eqref{energy_w11_inomega_z1} is proved.

{\bf STEP 3. Rescaling and $L^{\infty}$ estimates of $|\nabla w|$.}

Making the following change of variables on $\widehat{\Omega}_{\delta}(z')$ as in \cite{bll}
\begin{equation}\label{change-var}
\left\{
\begin{aligned}
&x'-z'=\delta y',\\
&x_n=\delta y_n,
\end{aligned}
\right.
\end{equation}
then $\widehat{\Omega}_{\delta}(z')$ becomes $Q_{1}$ of nearly unit size, where
\begin{equation}\label{def-Q1}
Q_r=\left\{y\in \R^n : -\frac{\va}{2\dt} +\frac{1}{\dt}h_2(\dt y^\prime +z^\prime)< y_n < \frac{\va}{2\dt} +\frac{1}{\dt}h_1(\dt y^\prime +z^\prime), \, |y^\prime |< r  \right\},
\end{equation}
for $r\leq1$, and the top and
bottom boundaries
become
\[
\Gamma^+_r=\left\{ y\in \R^n \, : \,
y_n=\frac{\varepsilon}{2\delta}+\frac{1}{\delta}h_{1}(\delta{y}'+z'),\quad|y'|<r\right\},\]
and
\[\Gamma^+_r=\left\{y\in \R^n \, : \,y_n=-\frac{\va}{2\delta}+\frac{1}{\delta}h_{2}(\delta{y}'+z'), \quad |y'|<r \right\}.\]
We denote
\[ \widetilde{w}(y^\prime, y_n):= w(\dt y^\prime + z^\prime , \dt y_n), \quad \widetilde{u}(y^\prime , y_n):=\bar{u}_1(\dt y^\prime +z^\prime, \dt y_n),  \quad (y^\prime, y_n)\in Q_1. \]

From \eqref{w20'}, we see that $\widetilde{w}$ satisfies
\begin{equation}\label{equ-w-divf-Q1}
\left\{ \begin{aligned}
-\Delta \widetilde{w}&=\Div (\nabla \widetilde{u})\quad &\text{in}\quad Q_1, \\
\widetilde{w}&=0 \quad &\text{on} \quad \Gamma_1^{\pm}.
\end{aligned}\right.
\end{equation}
Applying the De Giorgi-Nash estimates for \eqref{equ-w-divf-Q1}, see Lemma \ref{lem-infty-estimates} in the Appendix, we obtain 
\begin{equation}\label{esti-l-infty-poin}
\|\widetilde{w}\|_{L^\infty (Q_{1/2})} \leq C\left(\|\widetilde{w}\|_{L^2( Q_1)}+ [\nabla \widetilde{u}]_{\alp,\, Q_1} \right).
\end{equation}
By using the $C^{1,\alpha}$ estimates, Theorem \ref{lem-global-C1alp-estimates} with $\tilde{\mathbf{f}}= \nabla \widetilde{u}$ on $Q_{1/2}$,  we have
\begin{equation*}\label{esti-C1alp}
 \|\widetilde{w}\|_{C^{1, \, \alp}(Q_{1/4} )} \leq C\left( \|\widetilde{w}\|_{L^\infty( Q_{1/2})} + [\nabla \widetilde{u}]_{\alp,\, Q_{1/2}} \right)\leq C\left(\|\widetilde{w}\|_{L^2( Q_1)}+ [\nabla \widetilde{u}]_{\alp,\, Q_1} \right).
\end{equation*}
Combining with the  Poincar\'{e} inequality
$$\|\widetilde{w}\|_{L^2( Q_1)}\leq\,C\|\nabla \widetilde{w}\|_{L^2( Q_1)},$$
one has
\[\|\nabla \widetilde{w} \|_{L^\infty(Q_{1/4})} \leq  C \left(\|\nabla \widetilde{w}\|_{L^2( Q_1)}+ [\nabla \widetilde{u}]_{\alp, \, Q_1 } \right).\]

Rescaling back to the original region $\widehat{\om}_{\dt}(z^\prime)$,
\begin{equation}\label{ineq-scale-original}
\|\nabla w\|_{L^\infty( \widehat{\Omega}_{\delta/4}(z^\prime))}\leq \frac{C}{\delta}\left(\delta^{1-\frac{n}{2}}\|\nabla w\|_{L^2( \widehat{\Omega}_{\delta}(z^\prime))}+\delta^{1+\alp}[\nabla \overline{u}_1]_{\alp, \, \widehat{\Omega}_{\delta}(z^\prime) }\right).
\end{equation}
By virtue of \eqref{ineq-semi-holder-norm} and \eqref{energy_w11_inomega_z1}, we have, for $(z^\prime, x_n)\in \widehat{\Omega}_{\delta/4}(z^\prime)$ and $|z^\prime |\leq R_1$,
\begin{equation*}
|\nabla w(z^\prime, x_n)|\leq \|\nabla w\|_{L^\infty( \widehat{\Omega}_{\delta/4}(z^\prime))}\leq C\left(\delta^{-\frac{n}{2}}\cdot\delta^{\frac{n}{2}-\frac{1}{1+\alp}}+\delta^{\alp}\cdot\delta^{-\alp-\frac{1}{1+\alp}}\right)\leq C\delta^{-\frac{1}{1+\alp}}.
\end{equation*}
Thus, we finish the proof of Proposition \ref{prop-interior}.
\end{proof}

 Proposition \ref{prop-interior} also holds for $v_2$, defined in  \eqref{equ_v1}, if we choose an auxiliary function $\bar{u}_{2}=1-\bar{u}_{1}$ in $\Omega_{R_1}$.

\subsection{Proof of Lemma \ref{lemma-m-1-to-2}}

\begin{proof}[Proof of Lemma \ref{lemma-m-1-to-2}]
Recalling the definitions of $v_{1}$ and $v_{2}$ in \eqref{equ_v1}, one has
\begin{align*}
\begin{cases}
\Delta(v_{1}+v_{2}-1)=0& \mbox{in}~\Omega,\\
v_{1}+v_{2}-1=0& \mbox{on}~\partial{D}_{i},~i=1,2,\\
v_{1}+v_{2}-1=-1& \mbox{on}~\partial{D}.
\end{cases}
\end{align*}
By theorem 1.1 of \cite{llby}, we have \eqref{v1+v2_bounded1}. By the same reason,  \eqref{nabla_v3} also holds. It is easy to have \eqref{esti-C12} hold by the trace embedding theorem and $\|u\|_{H^{1}(\Omega)}\leq\,C$ (independent of $\varepsilon$).

For \eqref{esti-c1-c2}, we rewrite the decomposition \eqref{decomposition1_u} as follows
\begin{equation*}
u=(C_1-C_2) v_1 +C_2(v_1+v_2) +v_0.
\end{equation*}
From the third line of \eqref{equ-infty-interior}, we have
\begin{equation*}
\int_{\ptl D_1} \frac{\ptl u}{\ptl \nu}\Big|_+= (C_1-C_2)a_{11}+C_2(a_{11}+a_{12})+b_1=0,
\end{equation*} 
where
\begin{equation*}
a_{ij}:=\int_{\partial D_{i}}\frac{\partial v_{j}}{\partial\nu},\quad b_{i}:=\int_{\partial D_{i}}\frac{\partial v_{0}}{\partial\nu},\quad i,j=1,2.
\end{equation*}
Hence,
\begin{equation}\label{ineq-c1-c2-leq}
|C_1-C_2| \leq \frac{|C_2|\cdot |a_{11}+a_{12}|+|b_1|}{|a_{11}|}.
\end{equation}
By lemma 2.4 in \cite{bly1}, we have known that
\begin{equation}\label{a11+a12}
\frac{1}{C}\leq |a_{11}+a_{12}|\leq C\quad \text{and} \quad |b_i|\leq C\|\varphi\|_{L^\infty(\ptl D)}\quad \text{ for}\quad  i=1, \, 2.
\end{equation}

Now we calculate $a_{11}$. By using the Green's formula, 
\begin{equation}\label{equality-a11}
a_{11}=\int_{\partial D_{1}}\frac{\partial v_{1}}{\partial\nu}=\int_{\partial D_{1}}\frac{\partial v_{1}}{\partial\nu}v_{1}=-\int_{\Omega}|\nabla v_{1}|^{2}.
\end{equation}
We divide
\begin{equation*}
\int_{\Omega}|\nabla v_1|^2=\int_{\Omega_{R_1}}|\nabla v_1|^2+\int_{\Omega\setminus \Omega_{R_{1}}}|\nabla v_1|^2.
\end{equation*}
where it is easy to see from \eqref{ineq-nabla-v-out} that
\[\int_{\Omega\setminus \Omega_{R_{1}}}|\nabla v_1|^2\leq C.\]
Then, combining with the upper bound \eqref{nabla-v-i0}, a direct calculation yields
\begin{equation}\label{esti-rho}
\frac{1}{C\rho_{n,\,\alp}(\va)}\leq |a_{11}|\leq \frac{C}{\rho_{n,\,\alp}(\va)},
\end{equation}
where $\rho_{n,\,\alp}(\va)$ is defined in \eqref{def-rho-n-m}.
Thus, substituting this and \eqref{a11+a12} into  \eqref{ineq-c1-c2-leq}, we prove \eqref{esti-c1-c2}.
\end{proof}

\subsection{The Lower Bounds}\label{part 2 pf lower}
From the decomposition \eqref{decomposition_u2}, we write
\[\nabla u=(C_1-C_2)\nabla v_1+\nabla u_b,\quad \text{in}~ \Omega,\]
where
\[u_b:=C_2(v_1+v_2)+v_0,\]
verifying
\begin{equation*}
\begin{cases}
\Delta u_b=0\quad& \mathrm{in}~\Omega,\\
u_b=C_2& \mathrm{on}~\partial{D}_{1}\cup\partial{D_{2}},\\
u_b=\varphi& \mathrm{on}~\partial{D}.
\end{cases}
\end{equation*}
It follows from the third line of \eqref{equ-infty-interior} that
\begin{equation*}
-(C_1-C_2)a_{11}=\widetilde{b}_1,
\end{equation*}
where $\widetilde{b}_1:=\int_{\ptl D_1}\frac{\ptl u_b}{\ptl \nu}|_+$, which is a linear functional of $\varphi$. We observe from  \eqref{esti-rho}, that $a_{11}\neq 0$, so
\begin{equation}\label{equal-c1-c2}
|C_1-C_2|=\frac{|\widetilde{b}_1|}{-a_{11}}.
\end{equation}

For $n=2$, by using the same argument in \cite{lhg}, we have  $\widetilde{b}_1 \to \widetilde{b}_1^*$ as $\va \to 0$. Here $\widetilde{b}_1^*=\int_{\ptl D_1^*}\frac{\ptl u^*}{\ptl \nu}\big|_+$ is a blow-up factor and $u^*$ satisfies 
\begin{equation*}
\begin{cases}
\Delta u^*=0& \mathrm{in}~\Omega^*,\\
u^*=C^*& \mathrm{on}~\partial{D}_{1}^*\cup\partial{D_{2}}^*,\\
\int_{\ptl D_1}\frac{\ptl u^*}{\ptl \nu}\big|_++\int_{\ptl D_2^*}\frac{\ptl u^*}{\ptl \nu}\big|_+=0,\\
u_b=\varphi& \mathrm{on}~\partial{D},
\end{cases}
\end{equation*}
where $D_i^*=\{x\in \R^n : x+P_i \in D_i\}$ ($i=1, \ 2$), $\Omega^*=D\setminus \overline{D_1^*\cap D_2^*}$  and $C^*=\lim\limits_{\va \to 0}\frac{1}{2}(C_1+C_2)$. Thus, by using \eqref{esti-rho} and \eqref{equal-c1-c2}, if there exists $\varphi$ such that $\widetilde{b}_1^*[\varphi] \neq 0$, one has for small $\va>0$ 
\begin{align*}
|\nabla{u}(0',x_{n})|\geq&|C_{1}-C_{2}\|\nabla{v}_{1}(0',x_{n})|-|\nabla u_b(0',x_{n})|\\
\geq&\frac{\rho_{n,\,\alp}(\va) }{C\va}\cdot|\widetilde{b}_1^*[\varphi]|.
\end{align*}

For $n\geq 3$, it suffices to find a boundary data $\varphi$ such that $|\widetilde{b}_1[\varphi]|\geq \frac{1}{C}$ for some positive universal constant $C$, although $\widetilde{b}_1^*[\varphi] $ is not necessarily its limit. Then we have 
\[|\nabla u(x)|\geq \frac{1}{C\va}, \quad \text{for}~ x\in \overline{P_1P_2}.\]

\vspace{0.5cm}

\section{ Two Key Estimates in the Proof of Theorem \ref{thm1}}\label{pf of thm1}\label{sec4}

The key estimates in the proof of Theorem \ref{thm1} are the estimates of $|\nabla (v_{1}-\bar{u}_{1})|$, Proposition \ref{prop1}, and that of $|C_{1}-C_{2}|$, Proposition \ref{lemma-nabla-v1+v2}. Firstly, to prove Proposition \ref{prop1}, we follow the main idea in \cite{bll, llby} and list the main differences to show the role of $|\Sigma'|$ playing in such blow-up analysis. We emphasize that here the constants $C$ are independent of $|\Sigma'|$. For simplicity, we denote  $$d(x^\prime):=d_{\Sigma^\prime}(x^\prime)=\dist(x^\prime, \Sigma^\prime).$$

\subsection{Proof of Proposition \ref{prop1}}

First, we  denote
\begin{equation}\label{def_w}
w:=v_{1}-\bar{u}_{1}.
\end{equation}
From \eqref{equ_v1}, the definition of $\bar{u}_{1}$ and \eqref{def_w}, we have
\begin{equation}\label{w20}
\left\{
\begin{aligned}
-\Delta{w}&=\Delta\bar{u}_{1}\quad\mbox{in}~\Omega,\\
w &=0\quad\ \ \mbox{on}~\partial \Omega.
\end{aligned}\right.
\end{equation}
Similar as before, by virtue of the standard elliptic theory, one has that
\begin{equation*}
|w|+\left|\nabla{w}\right|\leq\,C,
\quad\mbox{in}~~ \Omega\setminus\Omega_{R_{1}}.
\end{equation*}
Recalling \eqref{ubar-out}, we have
\begin{equation}\label{ineq-nabla-v-out2}
|\nabla v_1(x)|\leq C, \qquad x\in \Omega \setminus \Omega_{R_1}.
\end{equation}
Thus,  to obtain \eqref{nabla_w_i0},
we only need to prove
\begin{equation}\label{section2 energyw1}
\left\| \nabla{w} \right\|_{L^{\infty}(\Omega_{R_{1}})}\leq\,C,
\end{equation}

Next, we mainly make use of an adapted version of the iteration technique developed in \cite{bll} to obtain the energy estimates in a small cube and then use $W^{2, p}$ estimates and the bootstrap argument to prove \eqref{section2 energyw1}. 

\begin{proof}[Proof of Proposition \ref{prop1}]

We divide into three steps.

{\bf STEP 1. The boundedness of the total energy: }

\begin{equation}\label{1energy_w}
\int_{\Omega}\left|\nabla{w}\right|^{2} \ dx \leq\,C.
\end{equation}

 Indeed, by using the maximum principle, we have $0<v_{1}<1$ in $\Omega$. Together with $|\bar{u}|\leq\,C$, we have
\begin{equation*}
\|w\|_{L^{\infty}(\Omega)}\leq\,C.
\end{equation*}
By a direct computation,
\begin{equation}\label{nabla2u_bar}
\Delta\bar{u}_{1}(x)=0,\quad x\in\Sigma,\quad |\Delta\bar{u}_{1}(x)|\leq\frac{C}{\varepsilon+d^{2}(x')},\quad\,x\in\Omega_{R_{1}}\setminus\Sigma.
\end{equation}
Then, multiplying the equation in \eqref{w20} by $w$ and integrating by parts, one has
\begin{align*}
\int_{\Omega}|\nabla{w}|^{2} \ dx
=\int_{\Omega}w\left(\Delta\bar{u}_{1}\right) \ dx
\leq\,\|w\|_{L^{\infty}(\Omega)}\left(\int_{\Omega_{R_{1}}\setminus\Sigma}|\Delta\bar{u}_{1}|+C\right)
\leq\,C.
\end{align*}

{\bf STEP 2. The local energy estimates:} 
\begin{equation}\label{energy_w_inomega_z1}
\int_{\widehat{\Omega}_{\delta(z')}(z')}\left|\nabla{w}\right|^{2}dx\leq
C\delta(z')^{n}.
\end{equation}

We adapt the iteration technique in \cite{bll} and give a 
unified iteration process for $0<|z'|<R_{1}$. For $0<t<s<R_{1}$, let $\eta$ be a cut-off function defined in \eqref{def-eta-x-prime}. Multiplying the equation in \eqref{w20} by $\eta^{2}w$ and integrating by parts
leads  to the Caccioppolli's inequality
\begin{align}\label{FsFt11}
\int_{\widehat{\Omega}_{t}(z')}|\nabla{w}|^{2} \ dx \leq\,\frac{C}{(s-t)^{2}}\int_{\widehat{\Omega}_{s}(z')}|w|^{2} \ dx
+(s-t)^{2}\int_{\widehat{\Omega}_{s}(z')}\left|\Delta\bar{u}_{1}\right|^{2} \ dx.
\end{align}

For $0<s<|z^\prime|\leq R_1/2$, similar to \eqref{energy_nabla_w11_square_in}, one has
\begin{equation}\label{int_w1}
\begin{aligned}
\int_{\widehat{\Omega}_{s}(z')}|w|^{2}\ dx
\leq\,C\delta(z^\prime)^2\int_{\widehat{\Omega}_{s}(z')}|\nabla{w}|^{2}\ dx\quad \text{if}~0<s<\frac{2|z^\prime|}{3}.
\end{aligned}
\end{equation}
Then, combining with \eqref{FsFt11} and \eqref{int_w1}, we have, for $0<t<s<\frac{2|z^\prime|}{3}$,
\begin{equation}\label{tildeF111_in}
F(t)\leq\,\left(\frac{c_{1}\delta(z^\prime)}{s-t}\right)^{2}F(s)+C(s-t)^{2}\int_{\widehat{\Omega}_{s}(z')}\left|\Delta\bar{u}_{1}\right|^{2} \ dx,
\end{equation}
where $c_1$ is the {\it universal constant}, we fix it now. $F(t)$ is defined in \eqref{denotation}.

Let $k=\left[\frac{\max\{\sqrt{\va} ,\, |z^\prime|\}}{4c_{1}\delta(z^\prime)}\right]$ and $t_{i}=\delta(z^\prime)+2c_{1}i\delta(z^\prime)$, $i=0,1,2,\cdots,k$. Take
$s=t_{i+1}$ and $t=t_{i}$ in \eqref{tildeF111_in}. It follows from \eqref{nabla2u_bar} that
\begin{align}\label{integal_Lubar11_in}
\int_{\widehat{\Omega}_{t_{i+1}}(z')}
\left|\Delta\bar{u}_{1}\right|^{2} \ dx &\leq \int_{|x^\prime-z^\prime|<t_{i+1}}\frac{C}{\delta(x^\prime)}\,dx^\prime
\leq\frac{Ct_{i+1}^{n-1}}{\delta(z^\prime)}\notag \\
&\leq C(i+1)^{n-1}\delta(z^\prime)^{n-2}.
\end{align}
An iteration formula follows from \eqref{tildeF111_in} and \eqref{integal_Lubar11_in},
\begin{equation*}
F(t_i)\leq \frac{1}{4}F(t_{i+1})+C(i+1)^{n-1}\delta(z^\prime)^n,\quad\,i=1,2,\cdots,k.
\end{equation*}
After iterating $k$ times, in view of \eqref{1energy_w}, we have
\begin{equation*}
\begin{aligned}
F(t_{0})
\leq (\frac{1}{4})^{k}F(t_{k})
+C\delta(z^\prime)^{n}\sum_{i=0}^{k-1}(\frac{1}{4})^{i}(i+1)^{n-1}
\leq C\delta(z^\prime)^{n}.
\end{aligned}
\end{equation*}
So \eqref{energy_w_inomega_z1} holds.

{\bf STEP 3.  Rescaling and $L^{\infty}$ estimates of $|\nabla w|$. }

Under the change of variables \eqref{change-var}, 
domain $\widehat{\om}_{\dt}(z^\prime)$ becomes $Q_1$, see \eqref{def-Q1}, with the top and bottom boundaries $\Gamma^\pm_1$. Further denote
\begin{equation*}
\widetilde{w} (y^\prime, y_n):= w(\dt y^\prime +z^\prime, \dt y_n),\quad\mbox{and}\qquad \widetilde{u} (y^\prime, y_n):=\bar{u}_1(\dt y^\prime +z^\prime, \dt y_n).
\end{equation*}
From \eqref{w20}, we see that $\widetilde{w}$ satisfies 
\begin{equation}\label{equation-w2p}
\left\{\begin{aligned}
-\Delta \widetilde{w}&= \Delta \widetilde{u} \quad  &\text{in} ~ Q_1,\\
\widetilde{w}&=0\quad &\text{on} ~ \Gamma^\pm_1.
\end{aligned}\right.
\end{equation}
By using $W^{2,p}$ estimates and the standard bootstrap argument for \eqref{equation-w2p} in $Q_1$, then rescaling back, the same as the step 1.3 in \cite{LX}, we obtain
\begin{equation}
\left\|\nabla{w}\right\|_{L^{\infty}(\widehat{\Omega}_{\delta/2}(z'))}\leq\,
\frac{C}{\delta}\left(\delta^{1-\frac{n}{2}}\left\|\nabla{w}\right\|_{L^{2}(\widehat{\Omega}_{\delta}(z'))}
+\delta^{2}\left\|\Delta\bar{u}_{1}\right\|_{L^{\infty}(\widehat{\Omega}_{\delta}(z'))}\right).
\label{AAA}
\end{equation}
Substituting \eqref{energy_w_inomega_z1} and \eqref{nabla2u_bar} into \eqref{AAA} yields
$$\left|\nabla{w}(z',z_{n})\right|\leq\frac{C\Big(\delta^{1-\frac{n}{2}}\delta^{\frac{n}{2}}+\delta\Big)}{\delta}\leq\,C,
\qquad\forall \
-\frac{\va}{2}+ h_{2}(z')<z_{n}<\frac{\varepsilon}{2}+h_1(z').$$
Thus, the estimate \eqref{nabla_w_i0} is established.
\end{proof}

We remark that Proposition \ref{prop1} also holds for $v_2$, the solution of \eqref{equ_v1}, if we  choose auxiliary function as $\bar{u}_{2}=1-\bar{u}_{1}$ in $\Omega_{R_1}$.

\subsection{ Proof of Proposition \ref{lemma-nabla-v1+v2}}\label{sub-proof-thm1}

The following lemma is a main difference with the analog in \cite{bly1}, which plays a key role in the blow-up analysis of $|\nabla u|$. 

\begin{lemma}\label{lemma_a11}
Under the hypotheses of Theorem \ref{thm1}, then for small $\va>0$, 
we have
\begin{align}\label{aii}
\frac{1}{C}\left(\frac{|\Sigma'|}{\varepsilon}+\frac{1}{\rho_n(\va)}\right)\leq-a_{ii}
\leq C\left(\frac{|\Sigma'|}{\varepsilon}+\frac{1}{\rho_n(\va)}\right),\quad i=1,2,
\end{align}
where $C$ is a {\it universal constant}, independent of $|\Sigma'|$.
\end{lemma}

In order to prove Lemma \ref{lemma_a11}, we need the following well-known property for bounded convex domains, which refers to the ellipsoid of minimum volume (see e.g. \cite[Theorem 1.8.2]{gu}).

\begin{lemma}
If $D\subset\mathbb{R}^{n}$ is a bounded convex set with nonempty interior and $E$ is the ellipsoid of minimum volume containing $D$ center at the center of mass of $D$, then
$$n^{-3/2}E\subset D\subset E,$$
where $rE$ denotes the $r$-dilation of $E$ with respect to its center.
\end{lemma}

Thus, for bounded convex $(n-1)$-dimensional domain $\Sigma'$, there exists a $E'$ such that
$$(n-1)^{-3/2}E'\subset \Sigma'\subset E'.$$
Denote the length of the longest principal semi-axis as $R_{0}$ and the length of the shortest principal semi-axis as $\widetilde{R}_{0}>0$. In order to show the role of $|\Sigma'|$ in the blow-up analysis of $|\nabla u|$, we suppose for simplicity that $\frac{R_{0}}{\widetilde{R}_{0}}\geq a$ for some $a>0$. Set $r_{0}=(n-1)^{-3/2}\widetilde{R}_{0}$. Obviously, $B'_{r_{0}}\subset \Sigma'\subset E'\subset B'_{R_{0}}$. Then, there exists a constant $C$, depending only on $n$ and $a$, such that
\begin{equation}\label{sigma}
|B'_{R_{0}}|\leq C|\Sigma'|.
\end{equation}

\begin{proof}[Proof of Lemma \ref{lemma_a11}]
Here, we only estimate $a_{11}$ for instance, since $a_{22}$ is similar. By virtue of the same reason of \eqref{equality-a11}, we have
\begin{align}\label{a11'}
-a_{11}=\int_{\Omega}|\nabla v_{1}|^{2}=\int_{\Omega\setminus \Omega_{R_1}}|\nabla v_{1}|^{2}+\int_{\Sigma}|\nabla v_{1}|^{2}+\int_{\Omega_{R_1}\setminus\Sigma}|\nabla v_{1}|^{2}.
\end{align}
For the first term in \eqref{a11'}, it is easy to see from (\ref{ineq-nabla-v-out2}) that
\begin{align}\label{R3}
\int_{\Omega\setminus \Omega_{R_1}}|\nabla v_{1}|^{2}\leq C.
\end{align}
For the second term, by (\ref{v1-bounded1}), 
\begin{align}\label{first_term}
\frac{|\Sigma'|}{C\varepsilon}\leq\int_{\Sigma'}\int_{0}^{\varepsilon}\frac{1}{C\varepsilon^2}dx_ndx'\leq\int_{\Sigma}|\nabla v_{1}|^{2}\leq\int_{\Sigma'}\int_{0}^{\varepsilon}\frac{C}{\varepsilon^2}dx_ndx'\leq\frac{C|\Sigma'|}{\varepsilon}.
\end{align}

For the last term in \eqref{a11'}, it is a little complicated. Using (\ref{v1-bounded1}) again, one has
\begin{align*}
&\int_{B'_{R_{1}}\setminus\Sigma'}
\int_{-\va/2+h_2(x')}^{\varepsilon/2+h_1(x')}\frac{1}{C(\varepsilon+d^2(x'))^{2}}dx_ndx'\\
&\leq
\int_{\Omega_{R_1}\setminus\Sigma}|\nabla v_{1}|^{2}\\
&\leq\int_{B'_{R_{1}}\setminus\Sigma'}
\int_{-\va/2+h_2(x')}^{\varepsilon/2+h_1(x')}\frac{C}{(\varepsilon+d^2(x'))^2}dx_ndx',
\end{align*}
which implies that
\begin{align}\label{a11_mid}
\int_{B'_{R_{1}}\setminus\Sigma'}\frac{dx'}{C(\varepsilon+d^2(x'))}\leq
\int_{\Omega_{R_1}\setminus\Sigma}|\nabla v_{1}|^{2}\leq\int_{B'_{R_{1}}\setminus\Sigma'}\frac{Cdx'}{\varepsilon+d^2(x')}.
\end{align}

Next, we divide into three cases by dimension to calculate the integral in \eqref{a11_mid}. Fist, if $n=2$, then $\Sigma'=(-R_{0},R_{0})$, and $d(x')=|x'|-R_{0}$. We can choose some constant $\tilde{\varepsilon}\in(0,1)$ depending only on $R_{1}$, such that for $0<\varepsilon<\tilde{\varepsilon}$,
\begin{align}\label{n=2}
\int_{R_0}^{R_1}\frac{dr}{C\left(\varepsilon+(r-R_0)^2\right)}=\frac{1}{C}\int_{0}^{R_1-R_0}\frac{dr}{\varepsilon+r^2}=\frac{1}{C\sqrt{\varepsilon}}\arctan\frac{R_{1}-R_{0}}{\sqrt{\varepsilon}}.
\end{align}
Inserting \eqref{R3}--\eqref{n=2} to \eqref{a11'}, we have, for small $\varepsilon>0$ (say, at least less than $(R_{1}-R_{0})^{2}$),
\begin{align*}
\frac{1}{C}\left(\frac{|\Sigma'|}{\varepsilon}+\frac{1}{\sqrt{\varepsilon}}\right)\leq -a_{11}\leq C\left(\frac{|\Sigma'|}{\varepsilon}+\frac{1}{\sqrt{\varepsilon}}\right),
\end{align*}
which implies \eqref{aii} for $n=2$.

For $n=3$, in view of \eqref{sigma}, we choose some constant $\tilde{\varepsilon}_1\in(0,1/e)$ such that for $0<\varepsilon<\tilde{\varepsilon}_1$, 
\begin{align*}
&\int_{B'_{R_{1}}\setminus\Sigma'}\frac{dx'}{\varepsilon+d^2(x')}\leq\int_{B'_{R_{1}}\setminus B'_{r_{0}}}\frac{dx'}{\varepsilon+dist^{2}(x',B'_{R_{0}})}\\
&\leq\int_{r_{0}}^{R_{0}}\frac{Cr}{\varepsilon}dr+\int_{R_{0}}^{R_{1}}\frac{Cr}{\varepsilon+(r-R_{0})^{2}}dr\\
&\leq\frac{C(R_0^2-r_{0}^{2})}{\varepsilon}+C\int_{R_{0}}^{R_{1}}\frac{r-R_{0}}{\varepsilon+(r-R_{0})^{2}}dr
+C\int_{R_{0}}^{R_{1}}\frac{R_{0}}{\varepsilon+(r-R_{0})^{2}}dr\\
&\leq C\left(\frac{R_0^2}{\varepsilon}+|\ln\varepsilon|+\frac{R_0}{\sqrt{\varepsilon}}\right)\leq C\left(|\ln\varepsilon|+\frac{|\Sigma'|}{\varepsilon}\right),
\end{align*}
where the Cauchy's inequality is used in the last inequality.

On the other hand, we pick a point $p\in\partial\Sigma'$, take a quadrant $Q$ outside $\Sigma'$, with $p$ as the vertex, $(R_{1}-R_{0})/2$ as the radius, and symmetric with the outword normal at $p$. Then, in the polar coordinates $\{p; r, \theta\}$ with $p$ as the center, for $x'\in Q$, we have $x'=p+(r\cos\theta, r\sin\theta)$, $\theta\in(-\frac{\pi}{4},\frac{\pi}{4})$, $r\in(0,(R_{1}-R_{0})/2)$, and $dist(x',\Sigma')\leq dist(x',p)$. There exists some small positive constant $\tilde{\varepsilon}\in(0,\tilde{\varepsilon}_1)$, depending only on $R_{1}$, such that for $0<\varepsilon<\tilde{\varepsilon}$, one has
\begin{align*}
\int_{B'_{R_{1}}\setminus\Sigma'}\frac{dx'}{\varepsilon+d^2(x')}&\geq\int_{Q}\frac{dx'}{\varepsilon+dist^{2}(x',p)}\\
&=\int_{-\frac{\pi}{4}}^{\frac{\pi}{4}}\int_{0}^{\frac{R_{1}-R_{0}}{2}}\frac{rdr}{\varepsilon+r^{2}} \geq\frac{1}{C}|\ln\varepsilon|.
\end{align*}
Substituting these two estimates above into \eqref{a11'}, together with \eqref{first_term} and \eqref{R3},  we have \eqref{aii} for $n=3$.

If $n\geq4$, by using \eqref{sigma} again, we have
\begin{align*}
\int_{B'_{R_{1}}\setminus\Sigma'}\frac{dx'}{\varepsilon+d^2(x')}&\leq
\int_{r_{0}}^{R_{0}}\frac{Cr^{n-2}}{\varepsilon}dr+\int_{R_{0}}^{R_{1}}\frac{Cr^{n-2}}{\varepsilon+(r-R_{0})^{2}}dr\\
&\leq\frac{C(R_0^{n-1}-r_{0}^{n-1})}{\varepsilon}+C\int_{0}^{R_1-R_0}\frac{(t+R_0)^{n-2}}{\varepsilon+t^2}dt\\
&\leq \frac{CR_0^{n-1}}{\varepsilon}+CR_0^{n-2}\int_{0}^{R_1-R_0}\frac{1}{\varepsilon+t^2}dt+C\int_{0}^{R_1-R_0}\frac{{t^{n-2}}}{\varepsilon+t^2}dt\\
&\leq C\left(\frac{R_0^{n-1}}{\varepsilon}+\frac{R_0^{n-2}}{\sqrt{\varepsilon}}+\int_{0}^{R_1-R_0}\frac{{t^{2}}}{\varepsilon+t^2}t^{n-4}dt\right),\\
&\leq C\left(\frac{|\Sigma'|}{\varepsilon}+1\right).
\end{align*}
For any $p\in\partial\Sigma'$, we also can construct a cone $Q\subset B'_{R_{1}}\setminus\Sigma'$ with $p$ as the vertex,  such that $dist(x',\Sigma')\leq dist(x',p)$ whenever $x'\in Q$. Then for small $\varepsilon>0$,
\begin{align*}
\int_{B'_{R_{1}}\setminus\Sigma'}\frac{dx'}{\varepsilon+d^2(x')}&\geq\int_{Q}\frac{dx'}{\varepsilon+dist^{2}(x',p)}
\geq\frac{1}{C}\int_{0}^{\frac{R_{1}-R_{0}}{2}}\frac{r^{n-2}}{\varepsilon+r^{2}}dr\geq\frac{1}{C}.
\end{align*}
Thus, \eqref{aii} holds for $n\geq4$. The proof of Lemma \ref{lemma_a11} is completed.
\end{proof}

\begin{remark}\label{rem aii}
Lemma \ref{lemma_a11} shows that if $|\Sigma'|>0$, by taking $\va>0$ small enough such that $\frac{1}{\rho_{n}(\va)}<\frac{|\Sigma'|}{\va}$, one has 
\begin{equation*}
\frac{|\Sigma'|}{C\varepsilon}\leq-a_{ii}\leq\frac{C|\Sigma'|}{\varepsilon},\quad i=1,2，
\end{equation*}
which leads the boundedness of $|\nabla u|$.
\end{remark}

\begin{proof}[Proof of Lemma \ref{lemma-nabla-v1+v2}] 
We prove \eqref{esti-c1-and-c2}. Indeed,
from \eqref{ineq-c1-c2-leq}, it suffices to estimate $|a_{11}|$. Combining with \eqref{ineq-c1-c2-leq} and \eqref{aii} of Lemma \ref{lemma_a11}, we obtain 
\begin{equation*}
|C_1-C_2|\leq \frac{C }{|a_{11}|} \cdot \|\varphi\|_{L^\infty(\ptl D)}\leq \frac{C\va}{|\Sigma^\prime|+\rho_n^{-1}(\va)\va}\cdot \|\varphi\|_{L^\infty(\ptl D)}.
\end{equation*}
The proof is completed.
\end{proof}

\vspace{0.5cm}

\section{Appendix : $C^{1,\alp}$ estimates and De Giorgi-Nash estimates}

\subsection{$C^{1,\alp}$ estimates } In this section, we shall use the Campanato's approach, see e.g. \cite{gm}, to prove Theorem \ref{lem-global-C1alp-estimates}. 

Let $Q$ be a  Lipschitz domain in $\R^n$, the  Campanato space $\mathcal{L}^{2,\lam}(Q)$, $\lam \geq 0$, is defined as follows
\begin{equation*}
\mathcal{L}^{2,\lam}(Q):=\Big\{u\in L^2(Q)~:~\sup_{x_0\in Q \atop \rho>0}\frac{1}{\rho^{\lam}}\int_{B_\rho(x_0)\cap Q}|u-u_{x_0, \rho}|^2 dx< +\infty \Big\},
\end{equation*}
where $u_{x_0, \rho}:=\frac{1}{|Q\cap B_\rho(x_0)|}\int_{Q\cap B_\rho(x_0)}u(x)\,dx$. It is endowed with the norm 
\begin{equation*}
\|u\|_{\mathcal{L}^{2,\lam}(Q)}:=\|u\|_{L^2(Q)}+[u]_{\mathcal{L}^{2, \lam}(Q)},
\end{equation*}
where the semi-norm $[\cdot]_{\mathcal{L}^{2,\lam}(Q)}$ is defined by
\begin{equation*}
[u]^2_{\mathcal{L}^{2, \lam}(Q)}:=\sup_{x_0\in Q \atop \rho>0}\frac{1}{\rho^{\lam}}\int_{B_\rho(x_0)\cap Q}|u-u_{x_0, \rho}|^2 dx.
\end{equation*}
It is known that if $n< \lam \leq n+2$ and $\alp=\frac{\lam-n}{2}$, the Campanato space $\mathcal{L}^{2,\lam}(Q)$ is equivalent to the H\"{o}lder space $C^{0, \alp}(Q)$.

We first recall a classical result in \cite{gm}.
\begin{theorem}(Theorem 5.14 in \cite{gm})\label{thm-514}
Let $Q$ be a bounded Lipschitz domain in $\R^n$, $n\geq 2$. Let $\widetilde{w}\in H^1(Q)$ be a solution for 
\begin{equation}\label{equ-w-divf-q14}
-\Delta \widetilde{w}=\Div \tilde{\mathbf{f}}\quad \text{in}~ Q, 
\end{equation}
with $\tilde{\mathbf{f}}\in C^{\alp}(Q, \R^n)$, $0<\alp<1$. Then $\nabla \widetilde{w} \in C^\alp(Q)$ and for $B_R:=B_R(x_0)\subset Q$, 
\begin{equation*}
\|\nabla \widetilde{w}\|_{\mathcal{L}^{2, n+2\alp}(B_{R/2})}\leq C\left(\|\nabla \widetilde{w}\|_{L^2(B_R)}+[\tilde{\mathbf{f}}]_{\mathcal{L}^{2, n+2\alp}( B_R)}\right),
\end{equation*}
where $C=C(n, \alp, R)$.
\end{theorem}

From the proof of Theorem \ref{thm-514} and the equivalence of H\"{o}lder space and Campanato space, we have the following interior estimates.

\begin{corollary}\label{lem-interior-estimate}
Under the hypotheses of Theorem \ref{lem-global-C1alp-estimates}. Let $\widetilde{w}$ be the solution of \eqref{equ-w-divf-q1}. Then for $B_R:=B_R(x_0)\subset Q$,
\begin{equation}\label{ineq-interior-estimate}
[\nabla \widetilde{w}]_{\alp, B_{R/2}}\leq C\left(\frac{1}{R^{1+\alp}}\|\widetilde{w}\|_{L^\infty( B_R)}+[\tilde{\bf{f}}]_{\alp, B_R}\right),
\end{equation}
where $C=C(n, \alp)$.
\end{corollary}

For the boundary estimate, we replace the ball $B_R(x_0)$ in \eqref{ineq-interior-estimate} by the half ball $B^+_R(x_0)=B_R(x_0)\cap \R^n_+$, where $x_0\in \partial \R^n_+$ and $ \R^n_+:=\{x\in \R^n: x_n>0\}$.  

\begin{corollary}\label{lem-halfball-estimate}
Let $\widetilde{w}$ be the solution of 
\begin{equation*}
\left\{
\begin{aligned}
- \Delta \widetilde{w} &= \Div \tilde{\mathbf{f}}\quad &\mbox{in} \quad &\R^n_+\\
\widetilde{w} &=0\quad &\mbox{on} \quad &\partial \R^n_+,
\end{aligned}\right.
\end{equation*}
where $\tilde{\mathbf{f}}\in C^{\alp} (\R^n_+, \R^n)$. Then for $x_0\in \partial \R^n_+$ and $B_R^+:=B^+_R(x_0)$,
\begin{equation}\label{ineq-halfball-estimate}
[\nabla \widetilde{w}]_{\alp, B^+_{R/2}}\leq C\left(\frac{1}{R^{1+\alp}}\|\widetilde{w}\|_{L^\infty( B_R^+)}+[\tilde{\mathbf{f}}\,]_{\alp, \, B_R^+}\right),
\end{equation}
where $C=C(n, \alp)$.
\end{corollary}

Now, we are in the position to prove Theorem \ref{lem-global-C1alp-estimates}.

\begin{proof}[Proof of Theorem \ref{lem-global-C1alp-estimates}]
Since $\Gamma \in C^{1,\alp}$, then for each point $x_0\in \Gamma$, there exists a neighbourhood $U$ of $x_0$ and a homeomorphism $\mathbf{\Psi}\in C^{1, \, \alp} (U)$ such that 
\begin{align*}
&\mathbf{\Psi} (U \cap Q )=\mathfrak{B}_1^+ = \{y\in \mathfrak{B}_1(0) \, : \, y_n>0 \},\\
&\mathbf{\Psi}  (U\cap\Gamma ) = \partial \mathfrak{B}_1^+ \cap \{y\in \R^n \, : \, y_n=0 \},
\end{align*}
where $\mathfrak{B}_1(0):=\{y\in \R^n : |y|< 1\}$. Under the transformation $y=\mathbf{\Psi} (x)=(\Psi_1(x), \cdots,\Psi_n(x))$, we denote
$$\mathcal{W}(y):=\widetilde{w}(\mathbf{\Psi} ^{-1}(y)),\quad \mathcal{F}(y):= \tilde{\mathbf{f}}(\mathbf{\Psi} ^{-1}(y)),$$ and
\begin{align*}
\mathcal{A}(y):=|\mathcal{J}(y)|(\mathcal{J}\mathcal{J}^{T})(y),\quad \mathcal{G}(y):=|\mathcal{J}(y)|\mathcal{J}(y), \quad  \mathcal{J}(y):=\frac{\partial(\Psi_1, \, \cdots \, \Psi_n)}{\partial (x_1, \, \cdots \, x_n)}\circ\mathbf{\Psi}^{-1}(y).
\end{align*}
Then \eqref{equ-w-divf-q1} becomes
\begin{equation}\label{equ-halfball-y}
-\Div_y \Big(\mathcal{A}(y) \nabla_y \mathcal{W}(y)\Big)= \Div_y\big( \mathcal{G}(y) (\mathcal{F}(y)-\mathfrak{a})\big),
\end{equation}
where $\mathfrak{a}\in \R^n$ is a constant vector to be determined later. Let $y_0= \mathbf{\Psi} (x_0)$, freeze the coefficients, and rewrite \eqref{equ-halfball-y} in the form 
\begin{equation}\label{equ-freeze-coefficients}
-\Div_y \Big(\mathcal{A}(y_0) \nabla_y \mathcal{W}(y)\Big) = \Div_y \Big( \big(\mathcal{A}(y)- \mathcal{A}(y_0)\big) \nabla_y \mathcal{W}(y)+ \mathcal{G}(y) (\mathcal{F}(y)-\mathfrak{a})\Big).
\end{equation}
Since $\mathbf{\Psi}$ is a homeomorphism, $\mathcal{A} (y_0)$ is positive definite. Then there exists a nonsingular constant matrix $\mathbf{P}$ such that $\mathbf{P}^T\mathcal{A} (y_0) \mathbf{P}=\mathbf{I}_n$, where $\mathbf{I}_n$ is the $n\times n$ identity matrix. Thus, under the transformation $z=\mathbf{P}^T y$, \eqref{equ-freeze-coefficients} becomes
\begin{align*}
-\Delta_z W(z) = \Div_z\Big( \mathbf{P}^T\big( \mathbf{A}(z)- \mathbf{A}(z_0)\big) \mathbf{P}\nabla_z W(z)+ \mathbf{P}^T \mathbf{G}(z) (\mathbf{F}(z)-\mathfrak{a})\Big) ,
\end{align*}
where $z_0=\mathbf{P}^T y_0$ and $W(z):=\mathcal{W}\big((\mathbf{P}^T)^{-1}z\big)$,
\begin{equation*}
\mathbf{A}(z):= \mathcal{A}\big((\mathbf{P}^T)^{-1}z\big), \quad
\mathbf{G}(z):= \mathcal{G}\big((\mathbf{P}^T)^{-1}z\big), ~\text{and}~
\mathbf{F}(z): =\mathcal{F}\big((\mathbf{P}^T)^{-1}z\big).
\end{equation*}
Then, by virtue of Corollary \ref{lem-halfball-estimate}, we have
\begin{align*}
\big[\nabla_z W\big]_{\alp,\, \mathcal{B}^+_{R/2}}\leq &C \left(\frac{1}{R^{1+\alp}}\| W\|_{L^\infty( \mathcal{B}^+_{R})}+\big[ \mathbf{P}^T \mathbf{G}(\mathbf{F}-\mathfrak{a})\big]_{\alp,\, \mathcal{B}^+_{R}} \right)\\
&\quad+C\big[\mathbf{P}^T \big (\mathbf{A}- \mathbf{A}(z_0)\big)\mathbf{P} \nabla_zW \big]_{\alp, \, \mathcal{B}^+_{R} },
\end{align*}
where $\mathcal{B}_{R}^+:=\{z\in \mathcal{B}_{R}(z_0) \, : \, z_n > 0\}$ and $ \mathcal{B}_{R}(z_0):=\{z\in \R^n : |z-z_0|<R\}$. Since $\mathbf{\Psi} \in C^{1, \, \alp}$, by taking 
\begin{equation*}
\mathfrak{a}=\mathbf{F}_{\mathcal{B}_R^+}:=\frac{1}{|\mathcal{B}_R^+|}\int_{\mathcal{B}_R^+}\mathbf{F}(z)\,dz,
\end{equation*}
we have
\begin{align*}
\big[ \mathbf{P}^T \mathbf{G}(\mathbf{F}-\mathbf{F}_{\mathcal{B}_R^+})\big]_{\alp,\, \mathcal{B}^+_{R}}
\leq C\left(  [\mathbf{F}]_{\alp,\, \mathcal{B}^+_{R}}+\|\mathbf{F}-\mathbf{F}_{\mathcal{B}_R^+}\|_{L^\infty(\mathcal{B}_R^+)}\right)\leq C   [\mathbf{F}]_{\alp,\, \mathcal{B}^+_{R}},
\end{align*}
and
\begin{align*}
\big[\mathbf{P}^T \big(\mathbf{A}- \mathbf{A}(z_0)\big)\mathbf{P} \nabla_z W \big]_{\alp, \, \mathcal{B}^+_{R} } \leq C \left( R^\alp [\nabla_z W]_{\alp, \, \mathcal{B}^+_{R} } + \|\nabla_z W\|_{L^\infty( \mathcal{B}^+_{R})}\right).
\end{align*}
By using the  interpolation inequality, one has 
\[\|\nabla_z W\|_{L^\infty(\mathcal{B}^+_{R})} \leq R^{\alp }[\nabla_z W]_{\alp, \, \mathcal{B}^+_{R} } + \frac{C}{R}\| W\|_{L^\infty( \mathcal{B}^+_{R})},\]
where $C=C(n)$. Hence, 
\begin{align}\label{ineq-tilde-W}
\big[\nabla_zW\big]_{\alp,\, \mathcal{B}^+_{R/2}}
\leq C \left(\frac{1}{R^{1+\alp}}\|W\|_{L^\infty( \mathcal{B}^+_{R})}+R^\alp [\nabla_z W]_{\alp, \, \mathcal{B}^+_{R} }+ [\mathbf{F}]_{\alp,\, \mathcal{B}^+_{R}} \right). 
\end{align}
Since $\mathbf{\Psi}$ is a homeomorphism and $\mathbf{P}$ is nonsingular, it follows that the norms in \eqref{ineq-tilde-W} defined on  $\mathcal{B}^+_{R}$ are equivalent to those on $\mathcal{N}=(\mathbf{P}^T\circ\mathbf{\Psi} )^{-1}(\mathcal{B}^+_{R})$, respectively.  Thus, rescaling back to the variable $x$, we obtain
\[\big[\nabla \widetilde{w}\big]_{\alp,\, \mathcal{N}^\prime} \leq C \left(\frac{1}{R^{1+\alp}}\|\widetilde{w}\|_{L^\infty(\mathcal{N})}+R^\alp [\nabla \widetilde{w}]_{\alp, \, \mathcal{N} }+ [\tilde{\mathbf{f}}]_{\alp,\, \mathcal{N}} \right),\]
where  $\mathcal{N}^\prime=(\mathbf{P}^T\circ\mathbf{\Psi} )^{-1}(\mathcal{B}^+_{R/2})$ and $C=C(n, \alp,\mathbf{\Psi} , \mathbf{P})$. Furthermore, there exists a constant $0<\sigma<1$ independent on $R$ such that $ B_{\sigma R}(x_0)\cap Q \subset \mathcal{N}^\prime$. 

Therefore, recalling that $\Gamma \subset \ptl Q$ is a boundary portion, for any domain $Q'\subset\subset Q\cup \Gamma$ and for each $x_0\in Q^\prime \cap\Gamma$, there exist  $\mathcal{R}_{0}=\mathcal{R}_{0}(x_0)$ and $C_0=C_0(n, \alp, x_0)$ such that
\begin{align}\label{ineq-boundary-estimates}
\big[\nabla \widetilde{w}\big]_{\alp,\, B_{\mathcal{R}_{0}}(x_0)\cap Q^\prime} 
\leq C_0 \left(\mathcal{R}_{0}^\alp [\nabla \widetilde{w}]_{\alp, \, Q^\prime}+\frac{1}{\mathcal{R}_{0}^{1+\alp}}\|\widetilde{w}\|_{L^\infty( Q)}+ [\tilde{\mathbf{f}}]_{\alp,\, Q} \right).
\end{align}

Applying the finite covering theorem to the collection of $B_{\mathcal{R}_{0}/2}(x_0)$ for all $x_0\in \Gamma \cap Q^\prime$, there exist finite  $B_{\mathcal{R}_{j}/2}(x_j)$, $j=1, 2, \, ...\, K$, covering  $  \Gamma \cap Q^\prime $. Let $C_{j}$ be the constant in \eqref{ineq-boundary-estimates} corresponding to $x_j$. Set 
\[\widehat{C}:=\max\limits_{1\leq j\leq K}\big\{C_{j}\big\},\quad\widehat{\mathcal{R}}:=\min\limits_{1\leq j\leq K}\big\{\frac{\mathcal{R}_{j}}{2}\big\}.\]
Thus, for any $x_0\in  \Gamma \cap Q^\prime$, there exists $j_0\in \{1, 2, \, ...\, ,K \}$ such that
$B_{\widehat{\mathcal{R}}}(x_0)\subset B_{\mathcal{R}_{j_0}}(x_{j_0})$ and
\begin{align}\label{ineq-local-estimates}
\big[\nabla \widetilde{w}\big]_{\alp,\, B_{\widehat{\mathcal{R}}}(x_0)\cap Q^\prime} 
\leq \widehat{C} \left(\widehat{\mathcal{R}}^\alp [\nabla \widetilde{w}]_{\alp, \, Q^\prime }+\frac{1}{\widehat{\mathcal{R}}^{1+\alp}}\|\widetilde{w}\|_{L^\infty( Q)}+ [\tilde{\mathbf{f}}]_{\alp,\, Q} \right).
\end{align}

Finally, we give the estimates  on $Q^\prime$. Let $\widetilde{C}$ be the constant in \eqref{ineq-interior-estimate} from Corollary \ref{lem-interior-estimate}. Let
\[\overline{C}:=\max\{\widehat{C}, \widetilde{C}\}\quad \text{and}\quad \overline{\mathcal{R}}:=\min\{(3\overline{C})^{-1/\alp}, \widehat{\mathcal{R}}\}.\]
For any $x^1, \, x^2\in Q^\prime $, there are three cases to occur:
\begin{itemize}
\item[(i)\,\,] $|x^1- x^2|\geq \frac{\overline{\mathcal{R}}}{2}$;

\item[(ii)\,] there exists $1\leq j_0\leq K$ such that $x^1, \, x^2\in B_{\overline{\mathcal{R}}/2}(x_{j_0})\cap Q^\prime $;

\item[(iii)]  $x^1, \, x^2\in B_{\overline{\mathcal{R}}/2}\subset
Q^\prime $.
\end{itemize}
For case (i), we have
\[\frac{|\nabla \widetilde{w}(x^1)-\nabla \widetilde{w}(x^2)|}{|x^1- x^2|^{\alp}}\leq \frac{4}{\overline{\mathcal{R}}^\alp}\|\nabla \widetilde{w}\|_{\infty, \, Q^\prime }.\]
For case (ii), it follows from (\ref{ineq-local-estimates}) that
\[
\begin{aligned}
\frac{|\nabla \widetilde{w}(x^1)-\nabla \widetilde{w}(x^2)|}{|x^1- x^2|^{\alp}}& \leq [\nabla \widetilde{w}]_{\alp, \, B_{\overline{\mathcal{R}}/2}(x_{j_0})\cap Q^\prime }\leq [\nabla \widetilde{w}]_{\alp, \, B_{\overline{\mathcal{R}}}(x_{j_0})\cap Q^\prime }\\
& \leq \overline{C} \left(\overline{\mathcal{R}}^\alp [\nabla \widetilde{w}]_{\alp, \, Q^\prime}
+\frac{1}{\overline{\mathcal{R}}^{1+\alp}}\|\widetilde{w}\|_{L^\infty(Q)}+ [\tilde{\mathbf{f}}]_{\alp,\, Q} \right).
\end{aligned}
\]
For case (iii), by using Corollary \ref{lem-interior-estimate}, one has
\[
\begin{aligned}
\frac{|\nabla \widetilde{w}(x^1)-\nabla \widetilde{w}(x^2)|}{|x^1- x^2|^{\alp}}& \leq [\nabla \widetilde{w}]_{\alp, \, B_{\overline{\mathcal{R}}/2}}
\leq \overline{C} \left(\frac{1}{\overline{\mathcal{R}}^{1+\alp}}\|\widetilde{w}\|_{L^\infty( Q)}+ [\tilde{\mathbf{f}}]_{\alp,\, Q} \right).
\end{aligned}
\]
Hence, in either case, we obtain
\begin{equation*}
[\nabla \widetilde{w}]_{\alp, \, Q^\prime } \leq \overline{C}\left(\overline{\mathcal{R}}^\alp [\nabla \widetilde{w}]_{\alp, \, Q^\prime  }+\frac{1}{\overline{\mathcal{R}}^{1+\alp}}\|\widetilde{w}\|_{L^\infty( Q)}+[\tilde{\mathbf{f}}]_{\alp,\, Q}\right)  +  \frac{4}{\overline{\mathcal{R}}^\alp}\|\nabla \widetilde{w}\|_{L^\infty( Q^\prime )}.
\end{equation*} 
By the interpolation inequality, see e.g. \cite[Lemma 6.32]{gt},
\[
\begin{aligned}
\frac{4}{\overline{\mathcal{R}}^\alp}\|\nabla \widetilde{w}\|_{L^\infty( Q^\prime )} &\leq \frac{1}{3} [\nabla \widetilde{w}]_{\alp, \, Q^\prime } + \frac{C}{\overline{\mathcal{R}}^{1+\alp}}\|\widetilde{w}\|_{L^\infty( Q^\prime )}\\
&\leq \frac{1}{3} [\nabla \widetilde{w}]_{\alp, \, Q^\prime } + \frac{C}{\overline{\mathcal{R}}^{1+\alp}}\|\widetilde{w}\|_{L^\infty(Q)},
\end{aligned}
\]
where $C=C(n, \, \alp)$. Since $\overline{\mathcal{R}}\leq (3\overline{C})^{-1/\alp}$, we get
\begin{equation}\label{ineq-global-estimates}
\big[\nabla \widetilde{w}\big]_{\alp,\, Q^\prime } \leq C\left( \|\widetilde{w}\|_{L^\infty( Q)}+[\tilde{\mathbf{f}}\,]_{\alp,\, Q}\right),
\end{equation}
where $C=C(n, \alp, Q^\prime, Q)$. By using the interpolation inequality, we obtain \eqref{ineq-global-C1alp-estimates}.
\end{proof}

\subsection{De Giorgi-Nash estimates}

In this subsection, we use  De Giorgi-Nash approach, see \cite[Theorem 8.15]{gt}, to obtain the $L^{\infty}$ estimate of $\widetilde{w}$  which is used to prove Proposition \ref{prop-interior}.

\begin{lemma}\label{lem-infty-estimates}
 Let $\widetilde{w}$ be the solution of \eqref{equ-w-divf-Q1} with $\nabla \widetilde{u} \in C^{\alp}(Q_1, \R^n)$, where $Q_1$ is defined in \eqref{def-Q1}. Then there exists a positive constant $C=C(n,\alp , Q_1)$ such that
\begin{equation*}
\|\widetilde{w}\|_{L^\infty(Q_{1/2})} \leq C \left(\|\widetilde{w}\|_{L^2(Q_1)}+[\nabla \widetilde{u}]_{\alp,\,  Q_1}\right).
\end{equation*}
\end{lemma}

\begin{proof}
For simplicity, we denote $\tilde{\mathbf{f}}:=\nabla \widetilde{u}$. Let $\bt \geq 1$, $N>k$, we choose a function $H\in C^1([k, \infty))$ by setting $H(t)=t^\bt -k^\bt$ for $t\in [k, N]$ and taking to be linear for $t\in [N, \infty)$.
We set $\psi= \widetilde{w}^+ +k$ and take
\[v=G(\psi)= \int_{k}^{\psi}|H^\prime (s)|^2ds\]
with $k=\|\tilde{\mathbf{f}}\|_{L^q(Q_1)}$ for $q > n$. Then, we take $\eta^2 v$ as a test function, where the cut-off function $\eta(y^\prime )$ satisfies $\eta(y^\prime )=1$ for $|y^\prime|\leq 1/2$  and $\eta(y^\prime )=0$ for $|y^\prime|=1$. Integrating by parts, using Young's inequality, and observing that $\int_{Q_1} \eta G(\psi)\nabla \eta \nabla \widetilde{w}^-\ dy=0$ and $G(s)\leq G^\prime(s)s$, one has
\[\int_{Q_1}\eta^2G^\prime(\psi)|\nabla \psi|^2\ dy\leq C\int_{Q_1}|\nabla \eta|^2 G^\prime(\psi)\psi^2\ dy + 4\int_{Q_1}\eta^2\frac{|\tilde{\mathbf{f}}|^2}{k^2} G^\prime(\psi)\psi^2 \ dy.\]
It follows from the definition of $G$ and the H\"{o}lder inequality that
\begin{equation} \label{ineq-l-infty-esti}
\int_{Q_1}|\eta\nabla H(\psi)|^2\ dy  
\leq C\int_{Q_1}|\nabla \eta|^2| H^\prime(\psi)\psi|^2\ dy+ C\left(\int_{Q_1}\left|\eta H^\prime(\psi)\psi\right|^{\frac{2q}{q-2}}\ dy \right)^{\frac{q-2}{q}}.
\end{equation}
By the interpolation inequality,  one has, for any $\lam >0$,
\begin{equation}\label{ineq-interpolation}
\|\eta H^\prime (\psi)\psi \|_{L^\frac{2q}{q-2}(Q_1)}\leq \lam \|\eta H^\prime (\psi)\psi\|_{L^\frac{2\hat{n}}{\hat{n}-2}( Q_1)}+\lam^{-\frac{\hat{n}}{q-\hat{n}}}\|\eta H^\prime (\psi)\psi\|_{L^2( Q_1)},
\end{equation}
where  $\hat{n}=n$ for $n>2$ and $\hat{n}\in (2, q) $ for $n=2$. Moreover, noting that $\eta H(\psi)\in H_0^1(Q_1)$, it follows from the Sobolev inequality that
\begin{equation}\label{ineq-sobolev}
\|\eta H(\psi)\|_{L^\frac{2\hat{n}}{\hat{n}-2}(Q_1)}\leq C\left(\int_{Q_1} |\eta \nabla H(\psi)|^2\ dy+\int_{Q_1}|H(\psi)\nabla \eta|^2\ dy\right)^{1/2},
\end{equation}
where $C=C(\hat{n})$. Then, combining with \eqref{ineq-l-infty-esti}--\eqref{ineq-sobolev}, we have
\begin{align*}
\|\eta H(\psi)\|^2_{L^\frac{2\hat{n}}{\hat{n}-2}( Q_1)} &\leq C\left(\lam^2\|\eta H^\prime(\psi)\psi\|^2_{L^\frac{2\hat{n}}{\hat{n}-2}( Q_1)}+\lam^{-\frac{2\hat{n}}{q-\hat{n}}}\|\eta H^\prime(\psi)\psi\|^2_{L^2(Q_1)}\right)\\
&\qquad+C\int_{Q_1}|\nabla \eta|^2\left(|H^\prime(\psi)\psi|^2+H^2(\psi)\right)\ dy.
\end{align*}
Take $\lam$ small such that
\begin{equation*}
\|\eta H(\psi )\|^2_{L^\frac{2\hat{n}}{\hat{n}-2}(Q_1)} \leq C\left(\int_{Q_1} |\eta H^\prime (\psi)\psi|^2\ dy+ \int_{Q_1} |\nabla \eta|^2 \big(|H^\prime (\psi)\psi|^2+H^2(\psi)\big)\ dy\right),
\end{equation*}
where $C=C(n, q)$. Letting $N \to \infty$, one has
\begin{equation}\label{ineq-psi-bt}
\|\eta \psi^\bt \|_{L^\frac{2\hat{n}}{\hat{n}-2}( Q_1)}\leq C\bt \left(\int_{Q_1} (\eta^2+|\nabla \eta|^2)\psi^{2\bt}\ dy\right)^\frac{1}{2}.
\end{equation}

Let $r_1, \, r_2$ be such that $1/2\leq r_1< r_2 \leq 1$ and set $\eta =1$ for $|y^\prime|\leq r_1$, $\eta =0$ for $|y^\prime|\geq r_2$ with $|\nabla \eta | \leq 2/(r_2-r_1)$. Writing $\chi= \hat{n}/(\hat{n}-2)$ in \eqref{ineq-psi-bt}, we have
\begin{equation*}
\|\psi\|_{L^{2\bt \chi}( Q_{r_1})}\leq \left(\frac{C\bt }{r_2-r_1}\right)^{\bt^{-1}}
\|\psi\|_{L^{2\bt}(Q_{r_2})}.
\end{equation*}
Then, iterating by $\bt= \chi^j$ and $r_j=\frac{1}{2}+\frac{1}{4^{j+1}}$, $j=0,\, 1, \cdots$, we have
\begin{equation*}
\begin{aligned}
\|\psi \|_{L^{2\chi^{j+1}}(Q_{r_{j+1}})} &\leq \left(\frac{C\chi^j}{r_j-r_{j+1}}\right)^{\chi^{-j}}
\|\psi \|_{L^{2\chi^j}(Q_{r_j})}\\
& \leq C^{(\chi^{-j}+\cdot\cdot\cdot+\chi^{-1}+1)}\cdot \chi^{(j\chi^{-j}+\cdot\cdot\cdot+2\chi^{-2}+\chi^{-1})}\|\psi\|_{L^2( Q_1)}.
\end{aligned}
\end{equation*}
Letting $j\to \infty$,
\[\|\psi\|_{L^\infty(Q_{1/2})}\leq C \|\psi\|_{L^2(Q_1 )},\]
where $C=C(n, q, Q_1)$. Recalling that $\psi = \widetilde{w}^++k$, a direct calculation yields
\begin{equation}\label{infty-esti-+}
\|\widetilde{w}^+\|_{L^\infty( Q_{1/2})}\leq C (\|\widetilde{w}\|_{L^2( Q_1)}+\|\tilde{\mathbf{f}}\|_{L^q( Q_1)}).
\end{equation}
We observe that \eqref{infty-esti-+} is also valid by replacing $\widetilde{w}$ with $-\widetilde{w}$. Thus, we have 
\begin{equation}\label{infty-esti-whole}
\|\widetilde{w}\|_{L^\infty( Q_{1/2})}\leq C (\|\widetilde{w}\|_{2, Q_1}+\|\tilde{\mathbf{f}}\|_{L^q( Q_1)}),
\end{equation}
where $q>n$ and $C=C(n, q, Q_1)$.

Since $\widetilde{w}$ still satisfies
\begin{equation*}
-\Delta \widetilde{w}= \Div (\tilde{\mathbf{f}}- \tilde{\mathbf{f}}_{Q_1}),
\end{equation*}
following the proof of  \eqref{infty-esti-whole}, we have
\begin{equation*}
\begin{aligned}
\|\widetilde{w}\|_{L^\infty(Q_{1/2})}&\leq C \left(\|\widetilde{w}\|_{L^2(Q_1)}+\|\tilde{\mathbf{f}}- \tilde{\mathbf{f}}_{Q_1}\|_{L^q(Q_1)}\right)\\
&\leq C \left(\|\widetilde{w}\|_{L^2(Q_1)}+[\tilde{\mathbf{f}}]_{\alp,\,  Q_1}\right),
\end{aligned}
\end{equation*}
where $C=C(n,\alp , Q_1)$. The proof is completed.
\end{proof}



\noindent
{\bf{\large Acknowledgements.}}
The authors would like to thank Dr. Hongjie Ju for valuable discussion. Y. Chen was partially supported by Postdoctoral Science Foundation of China No. 2018M631369. H.G. Li was partially supported by NSF in China No. 11571042, 11631002.


\end{document}